\documentclass[11pt]{amsart}
\usepackage{amssymb,amsmath}
\usepackage{amsthm, hyperref}
\usepackage[all]{xy}

\leftmargin  \leftmargini
\def\FF{{\mathbb F}}
\def\GG{{\mathbb G}}
\def\NN{{\mathbb N}}
\def\ZZ{{\mathbb Z}}

\def\G{{\mathcal G}}
\def\H{{\mathcal H}}
\def\I{{\mathcal I}}
\def\T{{\mathcal T}}
\def\U{{\mathcal U}}
\def\X{{\mathcal X}}

\def\fA{{\mathfrak A}}
\def\fH{{\mathfrak H}}
\def\fM{{\mathfrak M}}

\def\tF{\tilde{F}}
\def\tR{\tilde{R}}
\def\tS{\tilde{S}}
\def\tT{\tilde{T}}

\def\d{{\rm d}}
\def\m{{\mathfrak m}}
\def\n{{\mathfrak n}}
\def\eps{\varepsilon}
\def\id{{\rm id}}
\def\pr{{\rm pr}}
\def\iterate{{\rm I}}

\def\HD{{\rm HD}}
\def\ID{{\rm ID}}
\def\Der{{\rm Der}}

\def\HCon{{\bf HCon}}
\def\ICon{{\bf ICon}}
\def\IConint{{\bf ICon}_{\bf int}}
\def\Vect{{\bf Vect}}
\def\Mod{{\bf Mod}}
\def\Projmod{\text{\bf Proj-Mod}}
\def\cat#1{{\sf #1}}

\def\Dif{{\hat{\Omega}}}  
\def\difnabla{{}_\Dif\!\nabla}
\def\sdif{{\tilde{\Omega}}} 
\def\forget{\boldsymbol{\omega}}

\def\uob{\boldsymbol{1}}  
\def\dual{*}
\def\sep{{\rm sep}}
\def\gr{{\rm gr}}
\def\LT{{\rm LT}}

\def\Alg#1{$#1$-cga}
\def\Algen#1{$#1$-cgas}
\def\cga{cga}
\def\cgas{cgas}
\def\cgm{cgm}
\def\cgms{cgms}
\def\fe{formally \'etale}

\def\vect#1{\text{\boldmath $#1$\unboldmath}} 
\def\ie{i.\,e.}
\def\gener#1{\langle #1 \rangle}
\def\tgener#1{\ll\! #1 \! \gg}
\def\isom{\cong}
\def\congr{\equiv}
\def\restr#1{|_{#1}}
\def\ideal{\unlhd}
\def\betr#1{\lvert #1\rvert}

\DeclareMathOperator{\Ima}{Im}
\DeclareMathOperator{\Ker}{Ker}

\DeclareMathOperator{\Mor}{Mor}
\DeclareMathOperator{\Hom}{Hom}
\DeclareMathOperator{\Aut}{Aut}
\DeclareMathOperator{\End}{End}
\DeclareMathOperator{\ch}{char}
\DeclareMathOperator{\GL}{GL}
\DeclareMathOperator{\Mat}{Mat}
\DeclareMathOperator{\Gal}{Gal}
\DeclareMathOperator{\Quot}{Quot}

\DeclareMathOperator{\Spec}{Spec}
\def\uGal{\underline{\Gal}}
\def\iHom{\underline{\Hom}}

\def\markdef{\bf }
\def\mathdef{\boldmath }
\theoremstyle{plain}
\newtheorem{thm}{Theorem}[section]
\newtheorem{cor}[thm]{Corollary}
\newtheorem{lem}[thm]{Lemma}
\newtheorem{prop}[thm]{Proposition}

\newtheorem*{tannaka_thm}{Theorem \ref{tannakian_categories}}
\newtheorem*{galcorres_thm}{Theorem \ref{galois_correspondence}}

\theoremstyle{definition}
\newtheorem{defn}[thm]{Definition}
\newtheorem{exmp}[thm]{Example}
\newtheorem{rem}[thm]{Remark}

\newenvironment{notation}{{\bf Notation}\it}

\begin{document}

\title[Iterative connections and nonreduced Galois groups]{Galois
  theory for iterative connections and nonreduced Galois groups} 

\author{Andreas Maurischat}
\address{\rm {\bf Andreas Maurischat (n\'e R\"oscheisen)},
  Interdisciplinary Center for Scientific Computing, Heidelberg 
University, Im Neuenheimer Feld 368, 69120 Heidelberg, Germany }
\email{\sf andreas.maurischat@iwr.uni-heidelberg.de}



\begin{abstract}
This article presents a theory of {\it modules with iterative
  connection}. This theory is a generalisation of the theory of
  modules with connection in characteristic zero
to modules over rings of arbitrary characteristic. 
We show that these modules with iterative
  connection (and also the modules with {\it integrable} iterative
  connection) form a Tannakian category, assuming some nice properties
  for the underlying ring, and we show how this generalises to
  modules over schemes. 
We also relate these notions to 
stratifications on modules, as introduced by A.~Grothendieck
(cf. [BO78]) in order to extend integrable (ordinary) connections to
finite characteristic.
Over smooth rings, we obtain an  equivalence of stratifications and
integrable iterative connections. 
Furthermore, over a regular ring in positive
characteristic, we show that the
category of modules with integrable iterative connection is also
equivalent to the category of flat bundles as defined by D.~Gieseker
in [Gie75].

In the second part of this article, we set up a Picard-Vessiot theory
  for fields of solutions. For such a Picard-Vessiot extension, we
  obtain a Galois correspondence, which takes into account even
  nonreduced closed subgroup
 schemes  of the Galois group scheme on one hand
  and inseparable intermediate extensions of the Picard-Vessiot
  extension on the other hand. Finally, we compare our Galois
  theory with the Galois theory for purely inseparable field 
extensions given by S.~Chase in [Cha76].
\end{abstract}

\maketitle

\section{Introduction}

\noindent For characteristic zero, N. Katz described in \cite{katz} a
general setting of 
modules with connection to describe partial linear differential  
equations, and established a Galois theory from an
abstract point of view: He showed that -- under some assumptions on
the ring -- the category of modules with
 connection (and also that of modules with integrable connection)
 forms a neutral Tannakian category over the 
field of constants and neutral Tannakian categories are known to be
equivalent to categories of finite dimensional representations of
proalgebraic groups (see \cite{del_milne}). However, 
this theory works only in characteristic zero. 
This is mainly caused by the fact that in positive characteristic $p$, every
$p$-th power of an element in a ring is differentially constant.
A.~Grothendieck gave a notion of stratifications (cf. \cite{bert_ogus})
which generalises the notion of integrable connections to arbitrary
characteristic, and which turns out to be a ``good'' category. In
positive characteristic, 
a theorem of Katz (see \cite{gieseker}) shows
that over smooth schemes, modules with
stratifications are equivalent to flat bundles (or F-divided sheaves
as they are called in \cite{santos}), which enables Gieseker and
Dos Santos to obtain further properties of the fundamental group
scheme resp. the Tannakian group scheme.

In the first part of this article, we set up  a theory over rings of
arbitrary characteristic, which generalises 
the characteristic zero setting
not only in the integrable case, but also in the non-integrable case,
using so called iterative connections.
The integrable version, however, (so called modules with integrable
iterative connection) is again equivalent to flat bundles over a
regular ring in positive characteristic (cf. Section
\ref{proj_system}).

For obtaining this theory, differentials will be replaced by a
family of {\it higher differentials},
similar to the step from derivations to higher/iterative
derivations in positive characteristic (see for example \cite{mat_hart} and
  \cite{mat_put}).
In getting the right setting, the main idea is the following: For an
algebra $R$ over a perfect field $K$, regard an iterative derivation
on $R$ over $K$ (or more generally, a higher
derivation) not as a sequence of $K$-linear maps 
$\left(\partial^{(k)}:R\to R\right)_{k\in\NN}$ (as it is done in
\cite{hasse_schmidt}, \cite{mat_hart} etc.) but as a homomorphism
of $K$-algebras $\psi:R\to R[[T]]$ by summing up, in detail $\psi(r):=
\sum_{k=0}^\infty \partial^{(k)}(r)T^k$ ($\psi$ is often called the
Taylor series), and moreover regard the ring of power series $R[[T]]$ as a
completion of the graded $R$-algebra $R[T]$. This leads to the notion of
``\cgas'' 
(completions of graded algebras; cf. Section \ref{general_notation}),
which allows to
generalise the definition of a higher derivation and to obtain a
universal object $\Dif_{R/K}$ with a universal higher derivation
$\d_R:R\to \Dif_{R/K}$, replacing the module of differentials
$\Omega_{R/K}$ used in the classical theory (cf. Theorem \ref{diff_exists}).


In Section \ref{highder_modules}, we introduce the definition of a higher
connection on an 
$R$-module. Furthermore, we show that a finitely generated $R$-module
that admits such a higher connection is locally free, if $R$ is
regular and a finitely generated $K$-algebra (Corollary
\ref{automatically_projective}). At least in positive characteristic,
this is an improvement to the literature, since no integrability
condition is needed.
Although modules with higher connection might be interesting on
their own, our main concern are modules with so called iterative
connection and modules with integrable iterative connection
(cf. Section \ref{iterative_theory}), which are obtained by requiring
additional properties on the higher connection. One of the main results of
the first part is given in Section \ref{categorial}, namely

\begin{tannaka_thm}
Let $R$ be a regular ring over a perfect field $K$ and the localisation of a
finitely generated $K$-algebra, such that $\Spec(R)$ has a $K$-rational point.
Then the categories $\HCon(R/K)$, $\ICon(R/K)$ and $\IConint(R/K)$ of
$R$-modules with higher connection 
resp. iterative connection resp. integrable iterative connection are
neutral Tannakian categories over $K$.
\end{tannaka_thm}

The reason for considering iterative and integrable iterative
connections becomes clear in the next two sections. In Section
\ref{char_zero}, we have a look at characteristic zero. Here we show
that iterative connections on modules are in one-to-one correspondence
to ordinary connections, if $R$ is regular, and that the integrability
 conditions coincide via this correspondence. Hence the theory of
 modules with 
 (integrable) iterative connection really is a generalisation of
 modules with (integrable) connection in characteristic zero.

Section \ref{proj_system} is dedicated to the case of positive
characteristic. The main result here is the equivalence between
 the category $\IConint(R/K)$
and the
category of Frobenius compatible projective systems (Fc-projective
systems) over the ring $R$. (Again under the assumption that $R$ is
regular.) Essentially,
 Fc-projective systems over $R$ can be identified with flat
bundles over $\Spec(R)$
resp. F-divided sheaves on $\Spec(R)$.
Using the equivalence above, we can deduce from Corollary
\ref{automatically_projective} that for an Fc-projective system
$\{M_i\}_{i\in\NN}$, the $R$-module $M_0$ is locally free. This is a
slight improvement of \cite{santos}, Lemma 6, where the underlying
field $K$ is supposed to be algebraically closed. 

As mentioned in the beginning, stratifications on modules as
introduced by A.~Grothendieck
(cf. \cite{bert_ogus}) also generalise the notion of integrable (ordinary)
connections. At least if $R$ is
smooth over $K$ and $K$ is algebraically closed, in characteristic
zero as well as in positive characteristic, we can deduce from our results
 that the category of stratified modules and the category
$\IConint(R/K)$
are equivalent, using the equivalence between stratifications and
integrable connections on modules in characteristic zero
(cf. \cite{bert_ogus}, Thm. 2.15) respectively a theorem of Katz on the
equivalence of stratified modules and flat bundles in positive
characteristic (cf. \cite{gieseker}, Thm. 1.3). However, there is no
obvious direct correspondence between stratifications and integrable
iterative connections, and it is an open question whether there is any
correspondence at all, if $R$ is not smooth.

We conclude the first part of this paper by outlining a generalisation
of modules with higher 
connection to sheaves of modules with higher connection
(resp. (integrable) iterative 
connection) over schemes in Section \ref{diff_schemes}.

\medskip

In the second part (Sections \ref{pv-theory}, \ref{galois_corres} and
\ref{finite_inseparable}), we consider solution rings and solution
fields for modules with iterative connection, which we call pseudo
Picard-Vessiot rings (PPV-rings) resp. pseudo
Picard-Vessiot fields (PPV-fields), following the notion of classical
differential Galois theory. Indeed, the Picard-Vessiot theory given
here is set up in a more general context (namely for so called {\em
  $\theta$-fields}), so that it can be applied not only to modules
with iterative connection, but also to the iterative differential
modules as in \cite{mat_hart} and \cite{mat_put}. Given such a
PPV-ring $R$ over a $\theta$-field $F$, we obtain a Galois group
scheme $\G:=\uGal(R/F)$ defined over the constants $C_F$ of $F$
(cf. Prop. \ref{aut_is_affine}), and we show that $\Spec(R)$ is a $(\G
\times_{C_F} F)$-torsor (cf. Cor. \ref{spec_is_torsor}).
The main theorem of this part is the Galois correspondence, namely

\begin{galcorres_thm}{\bf (Galois correspondence)}\\
Let $R$ be a PPV-ring over some $\theta$-field $F$, $E:=\Quot(R)$ the
quotient field of $R$ and $\G:=\uGal(R/F)$ the Galois group
scheme of $R/F$.
\begin{enumerate}
\item There is an antiisomorphism of the lattices
$$\fH:=\{ \H \mid \H\leq\G \text{ closed subgroup schemes of }\G
\}$$
and
$$\fM:=\{ M \mid F\leq M\leq E \text{ intermediate $\theta$-fields} \}$$
given by 
$\Psi:\fH \to \fM,\H\mapsto E^{\H}$ and 
$\Phi:\fM \to \fH, M\mapsto \underline{\Gal}(RM /M)$.
\item If $\H\leq \G$ is normal, then
  $E^{\H}=\Quot(R^{\H})$ and 
  $R^{\H}$ is a PPV-ring over $F$ with Galois group scheme 
$\underline{\Gal}(R^{\H}/F)\isom \G/\H$.
\item If $M\in\fM$ is stable under the action of $\G$, then $\H:=\Phi(M)$
 is a normal subgroup scheme of $\G$, $M$ is a PPV-extension of $F$ and
$\underline{\Gal}(R\cap M /F)\isom \G/\H$.
\item For $\H\in \fH$, the extension $E/E^{\H}$ is separable if and
  only if $\H$ is reduced.
\end{enumerate}
\end{galcorres_thm}
\noindent Here, $R^{\H}$ resp. $E^{\H}$ denote functorial invariants of $R$
resp. $E$ under the action of the group functor $\H$ (cf. Section
\ref{galois_corres}).

Contrary to the Galois correspondence given by Matzat and van der Put
in \cite{mat_put} in the iterative differential case, our
correspondence takes into account not only reduced subgroup schemes
and intermediate iterative differential fields over which $E$ is
separable\footnote{This separability condition is missing in \cite{mat_put},
but has been added for example in \cite{amano} and \cite{heiderich}.},
but even the nonreduced subgroup schemes and those intermediate fields
over which $E$ is inseparable. By part iv) of the theorem, in this
general setting also the separability condition and the
reducedness condition correspond to each other.
Our Galois correspondence is quite similar to a Galois correspondence
given by M. Takeuchi for so called C-ferential fields between
intermediate C-ferential fields and closed subgroup schemes, although he
uses a different definition of PV-extension. The relation to this
correspondence is discussed in Remark \ref{rel_to_takeuchi}.
We conclude Section \ref{galois_corres} by some examples to enlight
our Galois correspondence (cf. Example \ref{exmp_group_schemes}).

In the last section, the Galois theory given here is compared with the
Galois theory for purely inseparable field
extensions given by S. Chase in \cite{chase}, who extended the theory of
N. Jacobson (cf. \cite{jacobson}) to Galois group schemes of arbitrary
exponents.

\section{Notation}\label{general_notation}

Throughout this article, $K$ denotes a perfect field, $R$ and
$\tilde{R}$ denote integral domains, which are finitely generated
$K$-algebras (or localisations of finitely generated $K$-algebras) and
$f:R\to \tilde{R}$ denotes a homomorphism of $K$-algebras. $M$ will
be a finitely generated $R$-module.

As mentioned in the introduction, we need the notion of ``completions
of graded algebras''. So let $\bigoplus_{i=0}^\infty B_i$ be a
graded $R$-algebra. Then the ideals
$I_k:=\bigoplus_{i=k}^\infty B_i$ form a filtration of the
algebra and one obtains a completion of $\bigoplus_{i=0}^\infty
B_i$ with respect to this filtration (cf. \cite{eisenbud}, Ch. 7.1). As an
$R$-module, this completion is isomorphic to $\prod\limits_{i=0}^\infty
B_i$.

\begin{defn}{\bf (\cgas)}
A commutative $R$-algebra $B$ is called a {\markdef completion of a
  graded algebra}, or a {\markdef
  \cga } for short, if $B$ is the completion of a graded
  $R$-algebra $\bigoplus_{i=0}^\infty B_i$ in the above sense.
We call $B_i$ the {\markdef $i$-th homogeneous component} of $B$. 
$B$ is called a {\markdef \Alg{\tR}}, if $B$ is a \cga{} with $B_0=\tR$.
The augmentation map will be denoted by $\eps:B\to B_0=\tR$. More generally,
  the projection map to the $i$-th homogeneous component will be
  denoted by $\pr_i:B\to B_i$.
\end{defn}

\begin{exmp} 
\begin{enumerate}
\item The ring of formal power series $R[[T]]$ is an \Alg{R}, with
$i$-th homogeneous component $R\cdot T^i$.
\item The ring $\tR$ is a \cga{} with $(\tR)_0=\tR$ and $(\tR)_i=0$
  for $i>0$. In particular,
the ring $R$ itself is the trivial \Alg{R} with $(R)_i=0$ for $i>0$.
\end{enumerate}
\end{exmp}

\begin{rem}
\begin{enumerate}
\item Similar to the notation of a power series as an infinite sum,
  elements of a \cga{} $B$ are denoted by $\sum_{i=0}^\infty
b_i$, where $b_i\in B_i$. This notation is also justified by the fact
that indeed, $\sum_{i=0}^\infty b_i$ is the limit of the sequence of
 partial sums $(\sum_{i=0}^n b_i)_{n\in\NN}$ in the given topology, or
in other words, $\sum_{i=0}^\infty b_i$ is a convergent series.
\item Since $\bigoplus_{k=0}^\infty B_k$ is dense in $B$, the continuous
extension of a given homomorphism of graded $R$-algebras is unique. By
a {\markdef homomorphism of \cgas}, we will always mean a homomorphism
that is induced by a graded homomorphism of the underlying graded algebras.
\item For two \cgas{} $B$ and $\tilde{B}$, we define the {\markdef
  tensor product} $B\otimes \tilde{B}$, namely the \cga{} with
  homogeneous components $(B\otimes \tilde{B})_k:=\sum_{i+j=k}
  B_i\otimes_R  \tilde{B}_j$.
\end{enumerate}
\end{rem}

We sometimes have to consider homomorphisms
between \cgas{} that aren't induced by homomorphisms of graded
algebras. So let  
$B$ and $\tilde{B}$ be \cgas{} and let $g:B\to \tilde{B}$ be
a continuous homomorphism of $K$-modules (or even $K$-algebras).
Then we define ``homogeneous components'' $g^{(i)}:B\to \tilde{B}$
($i\in\ZZ$) of $g$ to be the continuous homomorphisms of $K$-modules
given by
$$g^{(i)}\restr{B_j}:=\pr_{i+j}\circ g\restr{B_j}: B_j\to
\tilde{B}_{i+j}$$
for all $j\in\NN$ (set $\tilde{B}_{i+j}:=0$ for
$i+j<0$).
The $g^{(i)}$ uniquely determine $g$, because for all $b_j\in B_j$,
$\sum\limits_{i=-j}^\infty g^{(i)}(b_j)$ converges to $g(b_j)$.

Such a continuous homomorphism of $K$-modules $g:B\to \tilde{B}$ is
called {\markdef positive}, if $g^{(i)}=0$ for $i<0$, and we denote by
$\Hom_K^+(B,\tilde{B})$ the set of positive continuous homomorphisms
of $K$-modules from $B$ to $\tilde{B}$.
A short calculation shows that for \cgas{} $B$ and $\tilde{B}$,
a continuous homomorphism $g:B\to \tilde{B}$ is a
homomorphism of $K$-algebras if and only if the maps $g^{(k)}$ satisfy the
property
$$\forall\, k\in\NN, \forall\, r,s\in B:\quad g^{(k)}(rs) = \sum_{i+j=k}
g^{(i)}(r) g^{(j)}(s).$$
Furthermore, the monoid $(K,\cdot)$ acts on the set
$\Hom_K^+(B,\tilde{B})$ of 
positive continuous homomorphisms of $K$-modules by
$$(a.g)^{(i)}:=a^i\cdot g^{(i)}\quad (i\geq 0)$$
for all $a\in K$, $g\in \Hom_K^+(B,\tilde{B})$.
If $g$ is a homomorphism of algebras, then $a.g$ also is a
homomorphism of algebras.
Moreover, for $g\in \Hom_K^+(B,\tilde{B})$, $h\in
\Hom_K^+(\tilde{B},\tilde{\tilde{B}})$ and $a\in K$, we have
$$a.(h\circ g)=a.h \circ a.g,$$
\ie{} the action of $K$ commutes with composition.

\begin{defn}
For a \cga{} $B$, a {\markdef \cgm} over $B$ is the completion of a
graded module over the graded algebra $\bigoplus_{k=0}^\infty B_k$,
the completion taken by the topology induced from the grading.
In the same manner as for \cgas, we define homogeneous components of
\cgms, continuous homomorphisms between \cgms{} and homogeneous
components of those. There also
is an action of the monoid
$(K,\cdot)$ on the set of positive continuous homomorphisms between
two given \cgms. 
\end{defn}

\begin{rem}
Some special maps, that we will use here are
the higher derivations on rings and modules (cf. Sections \ref{highder} and
\ref{highder_modules}) -- maps in $\Hom_K^+(R,B)$
resp. $\Hom_K^+(M,B\otimes_R M)$ --, the extension $\d_\Dif$ of
the universal derivation to the algebra of higher differentials -- a
map in $\Hom_K^+(\Dif,\Dif)$ (cf. Section \ref{highder}) --, the
extensions of iterable higher derivations (cf. Section \ref{pv-theory})
as well as the extensions of higher connections on $M$ to maps in 
$\Hom_K^+(\Dif\otimes_R M,\Dif\otimes_R M)$ (cf. Section
\ref{highder_modules}) and the extensions of iterable higher
derivations on $M$ (cf. Section \ref{pv-theory}).
\end{rem}

\section{Higher Derivations and Higher Differentials}\label{highder}
In this section we explain the notion of higher derivations on rings and
modules. The definition used here is different from that introduced by
Hasse and Schmidt in \cite{hasse_schmidt}. In fact, it is a
generalisation, as we will show later on. This more general definition
is necessary to define the algebra of higher differentials as a
universal object. 

\begin{defn}
Let $B$ be a \Alg{\tR}. (As mentioned earlier $R$, $\tR$ denote integral
domains over $K$ together with a homomorphism of $K$-algebras $f:R\to
\tR$.)
A {\markdef higher derivation} of
$R$ to $B$ over $K$ is a homomorphism of $K$-algebras $\psi:R\to B$
satisfying $\eps \circ \psi=f:R\to B_0=\tR$.
The set of all higher derivations of $R$ to $B$ over $K$ will be
denoted by $\HD_K(R,B)$. In the special case of $B=R[[T]]$
(and $\tilde{R}=R$, $f=\id_R$) we set 
  $\HD_K(R):=\HD_K(R,R[[T]])$. 
\end{defn}

\begin{rem}\label{highder_formel}
\begin{enumerate}
\item Since a higher derivation $\psi\in\HD_K(R,B)$ can be regarded as a
positive continuous homomorphism from the \cga{} $R$ to  the \cga{}
$B$, the ``homogeneous
components'' of $\psi$ are denoted by $\psi^{(k)}:R\to B_k$ and for
every $r\in R$, we then have 
$\psi(r)=\sum_{k=0}^{\infty} \psi^{(k)}(r)$. (The right hand side is a
series that converges in the topology of $B$.)
\item Let $\psi\in \HD_K(R)$. Then since $\eps
  \circ \psi=\id_R$, the maps 
  $\psi^{(k)}:R\to R\cdot T^k\isom R$ are homomorphisms of 
$K$-modules and satisfy the following properties:
\begin{eqnarray}
\psi^{(0)} &=& \id_R \\
\forall\, k\in\NN, \forall\, r,s\in R:\quad \psi^{(k)}(rs) &=& \sum_{i+j=k}
\psi^{(i)}(r) \psi^{(j)}(s)
\end{eqnarray}
Furthermore, any sequence $\left(\partial^{(k)}\right)_{k\in\NN}$ of
$K$-module-homomorphisms
$\partial^{(k)}:R\to R$ satisfying
these two properties defines a  
higher derivation $\psi:R\to R[[T]]$ via $\psi(r):=\sum_{k=0}^\infty
\partial^{(k)}(r)T^r$.
\item As mentioned in the beginning, Hasse and
Schmidt defined a higher derivation to be a sequence
$\left(\partial^{(k)}\right)_{k\in\NN}$ as above. Hence our definition of a
higher derivation $\psi\in \HD_K(R)$ is equivalent to that of Hasse
and Schmidt. 
\end{enumerate}
\end{rem}

\begin{exmp}\label{phi_t_j}\rm
\begin{enumerate}
\item If the characteristic of $K$ is zero, then any $K$-derivation
  $\partial:R\to R$ (in the classical sense) gives rise to a higher
  derivation $\phi_\partial \in \HD_K(R)$ by 
$$\phi_\partial(r):=\sum_{k=0}^\infty \frac{1}{k!}\partial^k(r) T^k$$
(see also Section \ref{char_zero}).
\item For a polynomial algebra $R=K[t_1,\dots, t_m]$, every higher
derivation of $R$ into some \Alg{R} $B$ is given by an $m$-tuple
$(b_1,\dots,b_m)$ of elements of $B$ satisfying $\eps(b_j)=t_j$ for
all $j=1,\dots,m$.\\
The higher derivations
$\phi_{t_j}\in\HD_K(K[t_1,\dots,t_m])$ given by $\phi_{t_j}(t_i)=t_i$ for
$i\ne j$ and $\phi_{t_j}(t_j)=t_j+T$ play an important role. In the
classical context, $\phi_{t_j}^{(1)}$ is just the partial derivation with
respect to $t_j$. We therefore call $\phi_{t_j}$ the {\markdef higher
derivation with respect to $t_j$}. If $\tR$ is formally \'etale over
$K[t_1,\dots, t_m]$ (see Def. \ref{formally_etale} and Example
\ref{formally_etale_ex}), then the
$\phi_{t_j}\in \HD_K(K[t_1,\dots, t_m])$ uniquely extend to higher
derivations on $\tR$ by Proposition \ref{unique_extension}. These
derivations will also be referred to as 
{\markdef higher derivation with respect to $t_j$} and will also be denoted
by $\phi_{t_j}$.
\end{enumerate}
\end{exmp}

\begin{defn}\label{formally_etale} (cf. \cite{ega}, Def. 19.10.2)\\
Let $S, \tS$ be rings and $g:S\to \tS$ a homomorphism of rings. $\tS$
is called {\markdef formally \'etale} over $S$ if, for each surjective
homomorphism of rings $h:T\to \tT$ with nilpotent kernel, and all
homomorphisms $v:S\to T$ and $\tilde{v}:\tS \to \tT$ satisfying
$\tilde{v}\circ g=h\circ v$, there exists a unique
homomorphism $u:\tS \to T$ such that $u\circ g=v$ and $h\circ
u=\tilde{v}$, \ie{}, one obtains the following commutative diagram:

\centerline{\xymatrix{ 
\tS \ar[r]^{\tilde{v}} \ar@{-->}[dr]^{u}& \tT \\
S \ar[u]^{g} \ar[r]^v & T \ar[u]_h 
}}
\end{defn}

\begin{exmp}\label{formally_etale_ex}
\begin{enumerate}
\item As it is shown in \cite{ega}, Example 19.10.3(ii), localisations
  of a ring $S$
are formally \'etale over $S$.
\item  Every
  finite separable extension of a field $F$ is formally \'etale over $F$
  (cf. \cite{ega}, Ex. 19.10.3(iii)).
\end{enumerate}
\end{exmp}

A more general example is the following:
\begin{prop}
Let $S$ be a ring, let $\tS=S(y)$ be an
  extension of $S$ with minimal polynomial $m(X)\in S[X]$ of 
  $y$ and such that $m'(y)$ is invertible in $\tilde{S}$, where
$m'(X):=\frac{d}{dX}m(X)$. Then $\tS$ is formally \'etale over $S$.
\end{prop}

\begin{proof}
Let $h:T\to \tT$ be a surjective homomorphism with nilpotent kernel $I:=\Ker(h)$
and let $v:S\to T$ and $\tilde{v}:\tS \to \tT$ be as in Definition
\ref{formally_etale}. Since every lift $u$ of $v$ is given by the
image of $y$ in $T$, we have to show that there exists a unique
element $z\in T$ with $h(z)=v(y)=:\tilde{z}$ and $m(z)=0$. (By abuse
of notation, we also write $m(X)$ for the polynomial $v(m)(X)\in
T[X]$ and also for the polynomial $h(v(m))(X)\in \tT[X]$.)
We will show by induction that for each $k\geq 1$, there exists a
$z_k\in h^{-1}(\tilde{z})$ with $m(z_k)\in I^k$, and that $z_k$ is
unique modulo $I^k$ with this property. Since $I$ is nilpotent, this
proves the claim by choosing $k$ sufficiently large.

For $k=1$, any preimage $z_1$ of $\tilde{z}$ works, since
$I=\Ker(h)$. Now assume for given $k\geq 1$, there is a $z_k\in T$
satisfying $h(z_k)=\tilde{z}$ and $m(z_k)\in I^k$, which is unique
modulo $I^k$. Since $m'(y)$ is invertible in $\tS$, we have
$m'(\tilde{z})\in \tT^\times$ and so $m'(z_k)\in T^\times$, by
\cite{matsumura}, Ex. 1.1.

Using Taylor expansion, for  $\zeta\in I^k$, we have
$m(z_k+\zeta)\congr m(z_k)+m'(z_k)\zeta \mod{I^{k+1}}$.
Therefore $m(z_k+\zeta)\in I^{k+1}$ if and only if 
$\zeta\congr -m'(z_k)^{-1} m(z_k) \mod{I^{k+1}}$.
Since by hypothesis, $m(z_k)\in I^k$ and hence $-m'(z_k)^{-1}
m(z_k)\in I^k$, the element $z_{k+1}:=z_k-m'(z_k)^{-1} m(z_k)\in
T$ satisfies $h(z_{k+1})=\tilde{z}$ and $m(z_{k+1})\in I^{k+1}$, and $z_{k+1}$
is unique modulo $I^{k+1}$ with these properties, since $z_k$ was
unique modulo $I^k$.
\end{proof}

We return to higher derivations (again using the notation given at the
beginning of Section \ref{general_notation}).

\begin{prop}\label{unique_extension}
If $\tR$ is formally \'etale over $R$, then every higher derivation
$\psi\in\HD_K(R,B)$ to a \Alg{\tilde{R}} $B$ can be uniquely extended
to a higher derivation $\psi_e\in \HD_K(\tilde{R},B)$.
\end{prop}

The proof is almost identical to H. Matsumura's proof for the case
$B=\tR[[T]]$, so we will omit it here. (See \cite{matsumura},
Thm. 27.2; \fe{} is called $0$-etale there).

\begin{defn}\label{komp}
For $\psi\in \HD_K(R)$
we define a continuous endomorphism $\psi[[T]]$ on $R[[T]]$ by
$\psi[[T]](\sum_{i=0}^\infty a_i T^i):=\sum_{i=0}^\infty
\psi(a_i)T^i$.
(In fact, $\psi[[T]]$ is an automorphism.)
Using this we get a {\markdef multiplication} on $\HD_K(R)$ by
\begin{equation}\label{mult_formel}\psi_1\cdot
  \psi_2:=\psi_1[[T]]\circ \psi_2\in \HD_K(R)\end{equation} 
for $\psi_1,\psi_2\in \HD_K(R)$. This defines a group structure on
$\HD_K(R)$ (see \cite{matsumura},\S 27).
\end{defn}

The link to (ordinary) derivations is given by
\begin{prop}
For $\ch(K)=0$, the set $\Der(R):=\{\psi^{(1)} \mid \psi\in
\HD_K(R)\}$ is the $R$-module of derivations on $R$
(cf. Prop. \ref{der_itder}).
\end{prop}


We now turn our attention to the universal higher derivation:
\begin{thm}\label{diff_exists}
Up to isomorphism, there exists a unique \Alg{R} $\Dif_{R/K}$ (which
we call the {\markdef algebra of higher differentials})
together with a higher derivation $\d_R:R\to \Dif_{R/K}$ satisfying the
following universal property:\\
For each \Alg{\tilde{R}} $B$ and each higher derivation $\psi:R\to B$ there
exists a unique homomorphism of \Algen{\tR} $\tilde{\psi}:\tR\otimes
\Dif_{R/K}\to B$ with $\tilde{\psi} \circ (1\otimes \d_R) = \psi$.
In other words, $\Dif_{R/K}$ represents the functor
$\HD_K(R,-)$.
\end{thm}

\begin{proof}
We construct $\Dif_{R/K}$. Uniqueness is given by the universal
property.\\
Let $G=R[\d^{(k)}r \mid k\in\NN_+,r\in R]$ be the polynomial algebra over
$R$ in the variables $\d^{(k)}r$ and let the degree of  $\d^{(k)}r$ be
$k$. Define $I\ideal G$ to be the ideal generated by the union of the sets
\begin{align*}
&\{ \d^{(k)}(r+s)-\d^{(k)}r-\d^{(k)}s \mid k\in\NN_+; r,s\in R\}, \\
&\{ \d^{(k)}a\mid k\in\NN_+; a\in K\}\quad \text{ and} \\
&\{ \d^{(k)}(rs)-\sum_{i=0}^k \d^{(i)}r\cdot \d^{(k-i)}s\mid
  k\in\NN_+; r,s\in R\}, 
\end{align*}
where we set $\d^{(0)}r=r$ for all $r\in R$.
Therefore $I$ is a homogeneous ideal and we define $\Dif_{R/K}$ to be the
completion of the graded algebra $G/I$. We also define the higher
derivation $\d_R:R\to \Dif_{R/K}$ by $\d_R(r):=\sum_{k=0}^\infty
\d^{(k)}r$. (Here and in the following the residue class of
$\d^{(k)}r\in G$ in $\Dif_{R/K}$ will also be denoted by $\d^{(k)}r$.)\\
The universal property is seen as follows:
Let  $\psi:R\to B$ be a higher derivation. Then we define an
$R$-algebra-homomorphism  $g:G\to B$ by $g(\d^{(k)}r):= \psi^{(k)}(r)$
for all $k>0$ and $r\in R$. The properties of a higher derivation
imply that $I$ lies in the kernel of $g$, and therefore $g$ factors
through $\overline{g}:G/I\to B$. Hence, we get a homomorphism of algebras
$\Dif_{R/K}\to B$ by  extending $\overline{g}$ continuously and
therefore a homomorphisms of \Algen{\tR} $\tilde{\psi}:\tR\otimes
\Dif_{R/K}\to B$.
On the other hand, the condition $\tilde{\psi} \circ (1\otimes \d_R)
   = \psi$ forces this choice of $g$ and so $\tilde{\psi}$ is unique. 
\end{proof}

\begin{rem}
In \cite{vojta}, P.~Vojta defined an $R$-algebra $HS_{R/K}^\infty$ that
represents the 
functor $\tilde{R}\mapsto \HD_K(R,\tR[[T]])$ for any $R$-algebra
$\tR$. Actually, $HS_{R/K}^\infty$ is a graded algebra, and 
by construction $\Dif_{R/K}$ is just the completion of
$HS_{R/K}^\infty$ (cf. the construction of $HS_{R/K}^\infty$ in
\cite{vojta}, Def. 1.3). Hence, some properties of $\Dif_{R/K}$ can be
deduced from the properties of $HS_{R/K}^\infty$ given in \cite{vojta}.
\end{rem}

\begin{prop}\label{isom_of_diff}
\begin{enumerate}
\item[(a)] For every homomorphism of rings $f:R\to \tilde{R}$ there is
a unique homomorphism of \Algen{\tilde{R}} $Df:\tilde{R}\otimes
\Dif_{R/K}\to \Dif_{\tilde{R}/K}$ such that $\d_{\tilde{R}}\circ f= Df
\circ (1\otimes \d_R)$.
\item[(b)] If $\tilde{R}$ is \fe{} over $R$, then $Df$ is
an isomorphism.
\end{enumerate}
\end{prop}

\begin{proof}
Since $\d_{\tR}\circ f$ is a higher derivation on $R$, part (a)
follows from the universal property of $\Dif_{R/K}$.
Part (b) follows from \cite{vojta}, Thm. 3.6, where it is shown that
the homomorphism of the underlying graded algebras is an isomorphism
in this case.
\end{proof}

We consider three important examples.

\begin{thm}\label{dif_formula}
\begin{enumerate}
\item[(a)] Let $R=K[t_1,\dots ,t_m]$ be the polynomial ring in $m$
variables. Then 
$\Dif_{R/K}$ is the completion of the polynomial algebra
 $R[\d^{(i)}t_j \mid i\in\NN_+,j=1,\dots,m ]$.
\item[(b)] Let $F/K(t_1,\dots, t_m)$ be a finite separable field 
extension. Then $\Dif_{F/K}$ is the completion of the polynomial
 algebra
 $F[\d^{(i)}t_j \mid i\in\NN_+,j=1,\dots,m ]$.
\item[(c)] Let $(R,\m)$ be a regular local ring of dimension $m$, let
 $t_1,\dots, t_m$ generate $\m$ and assume that $R$ is a localisation
  of a finitely generated $K$-algebra and that $R/\m$ is a finite
  separable extension of $K$. Then $\Dif_{R/K}$ is the
completion of the polynomial algebra $R[\d^{(i)}t_j
  \mid i\in\NN_+,j=1,\dots,m ]$.
\end{enumerate}
\end{thm}

\begin{rem} 
The completion of such a polynomial algebra will be denoted by\linebreak
$R[[\d^{(i)}t_j\mid i\in\NN_+,j=1,\dots,m ]]$, although it is not
really a ring of power series, because it contains infinite sums of
different variables. 
\end{rem}

\begin{proof}
Part (a) is a direct consequence of \cite{vojta}, Prop. 5.1. Part (b)
then follows from Prop. \ref{isom_of_diff}(b), since by Example
\ref{formally_etale_ex}, $K(t_1,\dots, t_m)$ is \fe{} over
$K[t_1,\dots, t_m]$ and $F$ is \fe{} over $K(t_1,\dots, t_m)$.\\
It remains to prove (c):
We will show that $(R/\m)\otimes \Dif_{R/K}$ is isomorphic to
$(R/\m)[[\d^{(i)}t_j]]$. Then, since
$\Quot(R)\otimes \Dif_{R/K}$ is isomorphic to
$\Quot(R)[[\d^{(i)}t_j]]$ (Prop. \ref{isom_of_diff} and part
(b)), by \cite{hartshorne}, Ch.II, Lemma 8.9, it follows that
$(\Dif_{R/K})_k$ is a free $R$-module and that the residue classes of
any basis of  $(\Dif_{R/K})_k$ form a basis of $(R/\m \otimes
\Dif_{R/K})_k$. Hence we obtain $\Dif_{R/K}=R[[\d^{(i)}t_j]]$.\\ 
First, let $\psi:R\to B$ be a higher derivation of $R$ to an
\Alg{R/\m} $B$. Then for all $k\in \NN$ and  $r_1,\dots,
r_{k+1}\in\m$, we have
$$\psi^{(k)}(r_1\cdots r_{k+1})=\sum_{i_1+\dots +i_{k+1}=k}
\psi^{(i_1)}(r_1)\cdots \psi^{(i_{k+1})}(r_{k+1})=0,$$
since in each summand at least one $i_j=0$, and so
$\psi^{(i_1)}(r_1)\cdots \psi^{(i_{k+1})}(r_{k+1})\in \m
B=0$. Therefore $\psi^{(k)}$ (and $\psi^{(i)}$ for $i<k$) factors
through $R/\m^{k+1}$.\\
Next, since $R/\m$ is a finite separable extension of $K$, there is an
element $\bar{y}\in R/\m$ which generates the extension $K\subset
R/\m$. Let 
$g(X)\in K[X]$ be the minimal polynomial of $\bar{y}$, then starting
with an arbitrary representative $y\in R$ for $\bar{y}$, using the
Newton approximation $y_{n+1}=y_n- g(y_n)g'(y_n)^{-1}$, we obtain an
element $\tilde{y}_k\in R$ such that $g(\tilde{y}_k)\congr 0
\mod{\m^{k+1}}$ for given $k\in\NN$. (Note that the Newton
approximation is well defined and converges to a root of $g(X)$, due
to the fact that $\overline{g(y)}=g(\overline{y})=0\in  
R/\m$ and $\overline{g'(y)}=g'(\overline{y})\ne 0\in
R/\m$, so $g(y)\in \m$ and $g(y)\in R^\times$, as well as inductively
for all $n\in\NN$:  
$\bar{y}_{n+1}=\bar{y}_n=\bar{y}\in R/\m$, $g(y_{n+1})\in \m$ and
$g'(y_{n+1})\in R^\times$.) This proves that for all $k\in \NN$, the
ring $R/\m^{k+1}$ contains a subfield isomorphic to $R/\m$.\\
Now by \cite{matsumura}, Theorem 14.4, the associated graded ring
$\gr(R)$ of $R$ is isomorphic to the polynomial ring $(R/\m)[t_1,\dots,
t_m]$ and therefore we obtain
$\gr(R/\m^{k+1})\isom (R/\m)[t_1,\dots,
t_m]/\n^{k+1}$, where $\n$ is the ideal generated by 
$\{t_1,\dots, t_m\}$.
Furthermore, since $R/\m^{k+1}$ contains a subfield isomorphic to
$R/\m$, we see that the inclusion
$\iota_k:(R/\m)[t_1,\dots,t_m]/\n^{k+1}\to R/\m^{k+1}$ (given by the inclusion
$K[t_1,\dots,t_m]/\n^{k+1}\subset R/\m^{k+1}$ and $\bar{y}\mapsto
\tilde{y}_k$) is an isomorphism.\\
Hence, every higher derivation $\psi_\gr : \gr(R)\to B$ into an
\Alg{R/\m} $B$ induces a higher derivation $\psi_R:R\to B$ on $R$ by
$\psi_R^{(k)}:=\psi_\gr^{(k)} \circ \iota_k^{-1}$ ($k\in\NN$) and vice
versa. So 
$R/\m\otimes_R \Dif_{R/K} \isom R/\m \otimes_{\gr(R)} \Dif_{\gr(R)/K}  =
(R/\m)[[\d^{(i)}t_j]]$. 
\end{proof}

\begin{cor}\label{dif_projective}
Let $R$ be a finitely generated $K$-algebra which is a regular ring,
then the homogeneous components $(\Dif_{R/K})_k$ ($k\in\NN$) are
projective $R$-modules of finite rank.
\end{cor}

\begin{proof}
For every maximal ideal $\m\ideal R$, the localisation $R_\m$ fulfills
the conditions of Theorem \ref{dif_formula}(c). And so by Proposition
\ref{isom_of_diff}, $R_\m\otimes_R (\Dif_{R/K})_k\isom
(\Dif_{R_\m/K})_k$ is a free $R_\m$-module of finite rank. Hence by
\cite{eisenbud}, 
Thm. A3.2, $(\Dif_{R/K})_k$ is a projective $R$-module of finite rank.
\end{proof}

From now on we also write $\Dif$ for $\Dif_{R/K}$.

\begin{thm}\label{d_Dif}
For each $a\in K$, there is a continuous
homomorphism of $K$-algebras $a.\d_{\Dif}:\Dif\to \Dif$ defined by
$$a.\d_{\Dif}\left(\d_R^{(i)}r\right):= \sum_{j=0}^\infty
a^j\binom{i+j}{j}\d_R^{(i+j)}r$$ 
for all $i\in\NN$ and $r\in R$. The homomorphisms $a.\d_{\Dif}$
satisfy the following three conditions:
\begin{enumerate}
\item $a.\d_{\Dif}$ extends the higher derivation $a.\d_R:R\to \Dif$.
\item For all $a,b\in K$: $(a.\d_{\Dif})\circ (b.\d_{\Dif})=(a+b).\d_{\Dif}$.
\item $0.\d_{\Dif}=\id_{\Dif}$.
\end{enumerate}
For short, we will write $\d_{\Dif}$ instead of $1.\d_{\Dif}$ and
$-\d_{\Dif}$ instead of $-1.\d_{\Dif}$.
\end{thm}

\begin{proof}
It is not hard to check that $a.\d_\Dif$ is well defined. Then the
first and third statements are obvious and the second one is shown by
an explicit calculation using some combinatorial identities (see
\cite{roesch}, Theorem 2.5 for details). 
\end{proof}

\begin{rem}
\begin{enumerate}
\item By the second and the third property, we see that $a.\d_{\Dif}$
is actually an automorphism of $\Dif$ for all $a\in K$.
The endomorphisms $a.\d_{\Dif}$ play an important role in the iterative
theory, as will be seen in Section \ref{iterative_theory}.
\item From the definition, we see that $a.\d_{\Dif}$ is the image of
  $\d_{\Dif}$ under the action of $a\in K$, as given in Section \ref{general_notation}, thus the notation $a.\d_{\Dif}$.
\end{enumerate}
\end{rem}

\begin{prop}
For all $i,j\in\NN$ we have:
$$\d_{\Dif}^{(i)}\circ \d_{\Dif}^{(j)}=\binom{i+j}{i}
\d_{\Dif}^{(i+j)},$$
where $\d_{\Dif}^{(i)}$ denotes the $i$-th homogeneous component of
$\d_{\Dif}$ (cf. Section \ref{general_notation}). 
\end{prop}

\begin{proof}
For all $i,j\in\NN$ and $\omega\in\Dif$, the term $\left(\d_{\Dif}^{(i)}\circ
\d_{\Dif}^{(j)}\right)(\omega)$ is the coefficient of $a^ib^j$ in the
expression
$\left((a.\d_{\Dif})\circ(b.\d_{\Dif})\right)(\omega)$. Furthermore by Theorem
\ref{d_Dif}, $(a.\d_{\Dif})\circ(b.\d_{\Dif})=(a+b).\d_{\Dif}$ and so
$\left(\d_{\Dif}^{(i)}\circ \d_{\Dif}^{(j)}\right)(\omega)$ is the
coefficient of 
$a^ib^j$ in the expression $(a+b).\d_{\Dif}(\omega)=\sum_{k=0}^\infty
(a+b)^k \d_{\Dif}^{(k)}(\omega)$, \ie{}, equals
$\binom{i+j}{i}\d_{\Dif}^{(i+j)}(\omega)$. (For a finite field $K$,
one has to use the little trick explained in Remark
\ref{iteration_trick} to justify this conclusion.)
\end{proof}

\section{Higher Derivations on Modules and Higher
  Connections}\label{highder_modules} 

In the following, $M$ will denote a finitely generated $R$-module and
$B$ will be a \Alg{\tR}.

\begin{defn}
Let $\psi:R\to B$ be a higher derivation of $R$ to $B$ over $K$. A
{\markdef (higher) $\psi$-derivation} of $M$ is an additive map
$\Psi:M\to B\otimes_R M$ with  
$(\eps\otimes \id_M) \circ \Psi=f\otimes \id_M$ and
$\Psi(rm)=\psi(r) \Psi(m)$ for all $r\in R,m\in M$.
The set of $\psi$-derivations of $M$ is denoted by  $\HD(M,\psi)$.\\
A {\markdef higher connection} on $M$ is a $\d_R$-derivation $\nabla\in
\HD(M,\d_R)$. If $\nabla$ is a higher connection on $M$, for
any higher derivation $\psi\in \HD_K(R,B)$, we define the
$\psi$-derivation $\nabla_\psi$ on $M$ by 
$$\nabla_\psi:=(\tilde{\psi}\otimes \id_M)\circ \nabla:M\to
\Dif_{R/K}\otimes_R M\to B\otimes_R M\ $$
(cf. Remark \ref{komp_modul}(i)).\\
For all $a\in K$ we define an endomorphism $a.\difnabla:\Dif\otimes_R M\to
\Dif\otimes_R M$ by 
$$(a.\difnabla)(\omega \otimes x):= a.\d_{\Dif}(\omega)\cdot
(a.\nabla)(x)$$ 
for all $\omega\in\Dif$ and $x\in M$, i.e. 
$a.\difnabla = (\mu_{\Dif}\otimes \id_M)\circ (a.\d_{\Dif}\otimes
a.\nabla)$, where $\mu_{\Dif}$ denotes the multiplication map in $\Dif$.
\end{defn}

\begin{rem}\label{komp_modul}
\begin{enumerate}
\item For a given $\psi\in \HD_K(R,B)$, every homomorphism of
  \Algen{\tilde{R}} $g:B\to \tilde{B}$ induces a map
  $g_{*}:\HD(M,\psi)\to \HD(M,g\circ \psi), \Psi\mapsto (g\otimes
  \id_M) \circ \Psi$.\\
  Using the action of $K$ 
in this context, for every $\Psi\in\HD(M,\psi)$ and $a\in K$, we get an
  $(a.\psi)$-derivation $a.\Psi$.
\item Let $\psi_1,\psi_2\in \HD_K(R)$ and $\Psi_i\in \HD(M,\psi_i)$
  ($i=1,2$). Then as in Definition
  \ref{komp}, we have an automorphism $\Psi[[T]]$ of $R[[T]]\otimes
  M\isom M[[T]]$ and a product $\Psi_1 \cdot\Psi_2$, which
  is an element of $\HD(M,\psi_1 \psi_2)$.
\end{enumerate}
\end{rem}

Our next aim is to show that for a regular ring $R$ over a perfect
field $K$ every $R$-module that admits a higher connection is
projective (or - in
geometric terms - locally free). So we get the analogue of the
well-known fact in characteristic zero that a coherent sheaf equipped
with a holomorphic connection must be locally free, resp. the analogue
of the corresponding fact in the not necessarily commutative situation
given by Y.~Andr\'e in \cite{andre}, Cor. 2.5.2.2.

We first need the following lemma:

\begin{lem}\label{invertible_derivative}
Let $(R,\m)$ be a regular local ring of dimension $m$, let
 $t_1,\dots, t_m$ generate $\m$ and assume that $R$ is a localisation
  of a finitely generated $K$-algebra and that $R/\m$ is a finite
  separable extension of $K$.
Let $\phi_{t_j}\in\HD_K(R)$ ($j=1,\dots,m$) denote the higher
derivations with respect to $t_j$ (cf. Example \ref{phi_t_j} and
Thm. \ref{dif_formula}(c)).\\
Then for every $r\in R\setminus \{0\}$ there exist
$k_1,\dots,k_m\in\NN$ such that
\begin{enumerate}
\item[(1)]\ \centerline{ $\left( \phi_{t_m}^{(k_m)}\circ
        \dots\circ 
\phi_{t_1}^{(k_1)}\right)(r)\in R^{\times},\quad\quad\ $  }
\item[(2)] for all $l_1,\dots, l_m\in\NN$ with $l_j\leq k_j$ ($j=1,\dots,m$)
and $l_i<k_i$ for some $i\in\{1,\dots,m\}$,
$$\left( \phi_{t_m}^{(l_m)}\circ \dots\circ
\phi_{t_1}^{(l_1)}\right)(r)\not\in R^{\times}.$$
\end{enumerate}
\end{lem}

\begin{proof}
Let $r\in R\setminus \{0\}$. Choose $E\in\NN$ such that
$r\in\m^E$ and $r\not\in\m^{E+1}$. Then $r$ can (uniquely) be written
as
$$r=\sum_{\substack{\vect{e}=(e_1,\dots,e_m)\in\NN^m\\
\betr{\vect{e}}=E}} u_{\vect{e}} \vect{t}^{\vect{e}},$$
where $u_{\vect{e}}\in R$ and $u_{\vect{f}}\in R^{\times}$ for at
least one $\vect{f}=(f_1,\dots,f_m)$.\\
(We use the usual notation of multiindices:
$\betr{\vect{e}}=e_1+\dots+e_m$ and
$\vect{t}^{\vect{e}}=t_1^{e_1}\cdots t_m^{e_m}$.)
For arbitrary $\vect{l}=(l_1,\dots,l_m)\in\NN^m$ and
$\vect{e}\in\NN^m$ we have:
$$\left( \phi_{t_m}^{(l_m)}\circ \dots\circ
\phi_{t_1}^{(l_1)}\right)(\vect{t}^{\vect{e}})= 
\binom{e_1}{l_1}\cdots \binom{e_m}{l_m}\vect{t}^{\vect{e}-\vect{l}}=
\left\{ \begin{array}{cc} 
0 & \text{if } l_i>e_i \text{ for some }i, \\
1 & \text{if } l_j=e_j \text{ for all } j, \\
\in\m & \text{if } \betr{\vect{l}}<\betr{\vect{e}}.
\end{array}\right. $$
So if we choose $k_j=f_j$ ($j=1,\dots,m$), we get
\begin{eqnarray*}
&&\left( \phi_{t_m}^{(k_m)}\circ \dots \circ
\phi_{t_1}^{(k_1)}\right)(r)= \sum_{\betr{\vect{e}}=E} \left(
\phi_{t_m}^{(k_m)}\circ \dots\circ 
\phi_{t_1}^{(k_1)}\right)(u_{\vect{e}}\vect{t}^{\vect{e}})\\
&&\quad = \sum_{\betr{\vect{e}}=E} \sum_{\substack{0\leq l_j\leq k_j\\
j=1,\dots,m}} \left(\phi_{t_m}^{(k_m-l_m)}\circ \dots\circ 
\phi_{t_1}^{(k_1-l_1)}\right)(u_{\vect{e}})\left(\phi_{t_m}^{(l_m)}\circ
\dots\circ \phi_{t_1}^{(l_1)}\right)(\vect{t}^{\vect{e}}) \\
&&\quad \congr u_{\vect{f}}\cdot 1 \mod{\m}.
\end{eqnarray*}
So $\left( \phi_{t_m}^{(k_m)}\circ \dots\circ
\phi_{t_1}^{(k_1)}\right)(r)\in u_{\vect{f}}+\m \subset R^{\times}$,
and for all $\vect{l}\in\NN^m$ with $l_j\leq k_j$ ($j=1,\dots, m$) and
$l_i<k_i$ for some $i$, we have $\left( \phi_{t_m}^{(l_m)}\circ \dots\circ
\phi_{t_1}^{(l_1)}\right)(r)\in\m = R\setminus R^{\times}$, since
$\betr{\vect{l}}<E$.
\end{proof}

\begin{thm}\label{locally_free}
Let $(R,\m)$ be a regular local ring as in Lemma
\ref{invertible_derivative} and let $M$ be a finitely generated 
$R$-module with a higher connection $\nabla\in\HD(M,\d)$. Then $M$
is a free $R$-module.
\end{thm}

\begin{proof}
Let $\{x_1,\dots,x_n\}$ be a minimal set of generators of $M$. Assume
that $x_1,\dots,x_n$ are linearly dependent. Then there exists a
nontrivial relation $\sum_{i=1}^n r_i x_i=0$, with $r_i\in R$. Choose
$E\in\NN$ such that $r_j\in\m^E$ for all $j=1,\dots n$ and
$r_i\not\in\m^{E+1}$ for at least one $i$. Without loss of generality, let
$r_1\not\in\m^{E+1}$. Then choose $k_1,\dots,k_m\in\NN$ for $r_1$ as given
by the previous lemma. Then
\begin{eqnarray*}
0&=&\left( \nabla_{\phi_{t_m}}^{(k_m)}\circ \dots \circ
\nabla_{\phi_{t_1}}^{(k_1)}\right)\left(\sum_{i=1}^n r_i x_i\right)\\
&=& \sum_{i=1}^n \sum_{\substack{0\leq l_j\leq k_j\\
j=1,\dots,m}} \left(\phi_{t_m}^{(l_m)}\circ \dots\circ 
\phi_{t_1}^{(l_1)}\right)(r_i)
\left(\nabla_{\phi_{t_m}}^{(k_m-l_m)}\circ \dots \circ 
\nabla_{\phi_{t_1}}^{(k_1-l_1)}\right)(x_i) \\
&\congr&  \sum_{i=1}^n \left( \phi_{t_m}^{(k_m)}\circ \dots \circ
\phi_{t_1}^{(k_1)}\right)(r_i) \cdot x_i \mod{\m M}
\end{eqnarray*}
Since $ \left( \phi_{t_m}^{(k_m)}\circ \dots \circ
\phi_{t_1}^{(k_1)}\right)(r_1)\in R^{\times}$, we get $x_1\in
\gener{x_2,\dots,x_n} + \m M$, so $M=\gener{x_2,\dots,x_n} + \m M$ and
therefore by Nakayama's lemma $M=\gener{x_2,\dots,x_n}$, in
contradiction to the minimality of $\{x_1,\dots,x_n\}$.
So $x_1,\dots,x_n$ is a basis for $M$ and in particular $M$ is a free
$R$-module. 
\end{proof}

\begin{cor}\label{automatically_projective}
Let $K$ be a perfect field and let $R$ be a finitely generated
$K$-algebra which is a regular ring. Then every finitely generated
$R$-module $M$ with higher connection $\nabla$ is a projective
$R$-module.
\end{cor}

\begin{proof}
By the previous theorem, the localisations of $M$ at every maximal
ideal of $R$ are free. Hence $M$ is projective.
\end{proof}

\section{Iterative Derivations and Iterative
  Connections}\label{iterative_theory} 

\begin{defn}
A higher derivation  $\phi\in \HD_K(R)$ is called an {\markdef iterative
  derivation}, if 
$$\forall\, i,j\in\NN:\quad \phi^{(i)}\circ \phi^{(j)}=
\binom{i+j}{i}\phi^{(i+j)}\ .$$
The set of iterative derivations on $R$ is denoted by $\ID_K(R)$.\\
Let $M$ be an $R$-module and $\phi\in\ID_K(R)$. A higher
$\phi$-derivation  $\Phi\in\HD(M,\phi)$ is called an {\markdef
  iterative {\mathdef $\phi$}-derivation}, if 
$$\forall\, i,j\in\NN:\quad \Phi^{(i)}\circ \Phi^{(j)}=
\binom{i+j}{i}\Phi^{(i+j)}\ .$$
The set of iterative $\phi$-derivations is denoted by
$\ID(M,\phi)$.
\end{defn}

\begin{exmp}\label{phi_t_iterative}
If $R$ is the polynomial ring $K[t_1,\dots, t_m]$ or a \fe{} extension of
that ring, the higher derivations
$\phi_{t_j}$ with respect to $t_j$ (cf. Example \ref{phi_t_j}) are
iterative derivations. (For $K[t_1,\dots, t_m]$ this is obvious and
for extensions, it follows from Lemma \ref{id_structures}.)
\end{exmp}

\begin{rem}
Note that there is no sense in defining an iterative derivation
$\Phi\in \HD(M,\psi)$ for a non-iterative higher derivation
$\psi\in\HD_K(R)$. This is seen by using the characterisation of the
iterative derivations in Lemma \ref{it_eig} and Lemma
\ref{it_eig_modul}: For all $a,b\in K^{\sep}$, $(a.\Phi)(b.\Phi)$ is
an $(a.\psi)(b.\psi)$-derivation, whereas $(a+b).\Phi$ is an
$(a+b).\psi$-derivation.
\end{rem}

\begin{lem}\label{it_eig} {\bf (Characterisation of iterative derivations)}\\
Let $\psi\in \HD_K(R)$ be a higher derivation. Then the following
conditions are equivalent:
\begin{enumerate}
\item[{\rm (i)}] $\psi$ is iterative,
\item[{\rm (ii)}] $\tilde{\psi}\circ \d_{\Dif} = \psi[[T]]\circ
\tilde{\psi}$, $\quad$ (see Thm. \ref{diff_exists} for the definition
of $\tilde{\psi}$.)
\item[{\rm (iii)}] For all $a\in K$: $\tilde{\psi}\circ (a.\d_{\Dif}) =
  (a.\psi[[T]])\circ \tilde{\psi}$.
\end{enumerate}
If $K$ is an infinite field, then this is also equivalent to
\begin{enumerate} 
\item[{\rm (iv)}] For all  $a,b\in K$:  $(a.\psi)(b.\psi)= (a+b).\psi$,
\end{enumerate}
whereas for arbitrary $K$ the conditions {\rm (i)}-{\rm (iii)} only
imply condition {\rm (iv)}. 
\end{lem}

\begin{proof}
For $a\in K$, $r\in R$ and $i\in\NN$ we have:
$$\tilde{\psi}\circ (a.\d_{\Dif})(\d^{(i)}r)=
\tilde{\psi}\left(\sum_{j=0}^\infty a^j \binom{i+j}{j}
\d^{(i+j)}r\right) 
= \sum_{j=0}^\infty a^j \binom{i+j}{j} \psi^{(i+j)}(r)T^{i+j}$$
and
$$ (a.\psi[[T]])\circ \tilde{\psi}(\d^{(i)}r)=
 a.\psi[[T]]\left(\psi^{(i)}(r)T^i\right) 
= \sum_{j=0}^\infty a^j \psi^{(j)}(\psi^{(i)}(r))T^{i+j}.$$
So by comparing the coefficients of $T^{i+j}$ one sees that condition (iii) is fulfilled if and only if $\tilde{\psi}\circ
(a.\d_{\Dif}) = (a.\psi[[T]])\circ \tilde{\psi}$ holds for arbitrary
$a\in K\setminus \{ 0\}$ (e.g. $a=1$, i.\,e. condition (ii)). Moreover,
 this is fulfilled if and only if
for all $i,j\in \NN$ we
have $\psi^{(j)}\circ \psi^{(i)}=\binom{i+j}{j} \psi^{(i+j)}$,
i.\,e. $\psi$ is iterative.\\
Furthermore, we get for all  $a,b\in K$:
$$\left((a.\psi)(b.\psi)\right)^{(k)}=\sum_{i+j=k}(a.\psi)^{(i)}\circ
(b.\psi)^{(j)} 
= \sum_{i+j=k} a^i b^j \psi^{(i)}\circ \psi^{(j)},$$
since $b\in K$, and 
$$((a+b).\psi)^{(k)} =(a+b)^k \psi^{(k)}
 = \sum_{i+j=k} a^i b^j \binom{i+j}{i} \psi^{(i+j)}.$$
So if $\psi$ is iterative, condition (iv) is fulfilled. If $\#K=\infty$,
we obtain from condition (iv) that $\psi$ is iterative
by comparing the coefficients of $a^i$. 
\end{proof}

If $\#K<\infty$, the following example shows
that condition (iv) doesn't imply the others.

\begin{exmp}
Condition (iv) is in fact weaker if $K$ is finite. If for example
$K=\FF_q$ and $R=\FF_q[t]$, then $\psi\in \HD_K(R)$ defined by
$\psi(t)= t+1\cdot T^{2q-1}$ is not iterative, since
$$(2q-1) \psi^{(2q-1)}(t)=2q-1\ne 0 =
\psi^{(2q-2)}\left(\psi^{(1)}(t)\right).$$ 
On the other hand,
for all $a\in \FF_q$ we have $a^{2q-1}=a$ and so 
\begin{eqnarray*}
\bigl((a.\psi)(b.\psi)\bigr)^{(k)}(t)&=& \sum_{i+j=k}
a^ib^j\psi^{(i)}(\psi^{(j)}(t)) = a^k
\psi^{(k)}(t)+a^{k-2q+1}b^{2q-1}\psi^{(k-2q+1)}(1)\\ 
&=&
\left\{ \begin{array}{cc}
t & k=0 \\  a^{2q-1}+b^{2q-1}=(a+b)^{2q-1} & k=2q-1\\ 0 & \text{otherwise}
\end{array} \right\} = \bigl((a+b).\psi\bigr)^{(k)}(t)
 \end{eqnarray*}
for all $a,b\in K=\FF_q$.
\end{exmp}

\begin{rem}\label{iteration_trick}
Condition {\rm (iv)} is very useful for calculations -- even if $K$ is
finite. If one has to
show that some higher derivation $\psi\in\HD_K(R)$ is iterative, one
can often use the following trick:\\
Let $\tR:=K^{\sep}\otimes_{\bar{K}\cap R} R$ be the maximal separable
extension of $R$ by constants. Then by Proposition
\ref{unique_extension} the higher derivation $\psi$ uniquely extends
to a higher derivation $\psi_e \in \HD_K(\tR)=\HD_{K^\sep}
(\tR)$. Since $\# 
K^{\sep}=\infty$, we can use condition (iv) to show that $\psi_e$ is
iterative and therefore $\psi$ is iterative.\\
Whenever it will be shown that for all $a,b\in K^{\sep}$,
$(a.\psi)(b.\psi)=(a+b).\psi$, this trick will be used, although we
won't mention it explicitly.
\end{rem}

\begin{lem}\label{it_eig_modul}
Let $\phi\in \ID_K(R)$ be an iterative derivation and
 $\Psi\in \HD(M,\phi)$ be a $\phi$-derivation.
Then $\Psi$ is iterative if and only if for all $a,b\in K^{\sep}$ the
 identity  $(a.\Psi)(b.\Psi)= (a+b).\Psi$ holds. 
\end{lem}

\begin{proof}
Analogous to the proof of Lemma \ref{it_eig}.
\end{proof}

The next lemma states some structural properties of $\ID_K(R)$.

\begin{lem}\label{id_structures}
\begin{enumerate}
\item If two iterative derivations $\phi_1,\phi_2\in\ID_K(R)$ commute,
  i.\,e. they satisfy
$\phi_1^{(i)}\circ \phi_2^{(j)}=\phi_2^{(j)}\circ \phi_1^{(i)}$ for all
  $i,j\in\NN$, then $\phi_1\phi_2$ is again an
  iterative derivation.
\item $\ID_K(R)$ is stable under the action of $K$.
\item If $\tR\supset R$ is a ring extension such that every higher
  derivation on $R$ uniquely extends to a higher derivation on $\tR$
  (e.\,g. if $\tR$ is \fe{} over $R$), then  the
  extension $\phi_e\in\HD_K(\tR)$ of an iterative derivation
  $\phi\in\ID_K(R)$ is again iterative.
\end{enumerate}
\end{lem}

\begin{proof}
A short calculation shows that all occurring higher
derivations satisfy condition (iv) of Lemma \ref{it_eig}, hence are iterative.
\end{proof}

We have already seen that $\d_{\Dif}$ satisfies the condition
$\d_{\Dif}^{(i)}\circ \d_{\Dif}^{(j)}=\binom{i+j}{i}\d_{\Dif}^{(i+j)}$
and that for iterative derivations $\phi\in\ID_K(R)$ we have the ``same''
condition $\phi^{(i)}\circ \phi^{(j)}=\binom{i+j}{i}
\phi^{(i+j)}$. This motivates the following definition of an iterative
connection.

\begin{defn}
A higher connection $\nabla$ on $M$ is called an {\markdef iterative
  connection} if the identity 
$${}_\Dif\!\nabla^{(i)}\circ \difnabla^{(j)}=\binom{i+j}{i}
\difnabla^{(i+j)}$$ 
holds for all $i,j\in\NN$. (As defined in Section \ref{general_notation} for the general case,
$\difnabla^{(i)}$ denotes the part of $\difnabla$ that ``increases
degrees by $i$''.)\\
An iterative connection $\nabla$ on $M$ is called an {\markdef
integrable iterative connection} 
if for all higher derivations
$\psi_1,\psi_2\in\HD_K(R)$ we have
$\nabla_{\psi_1\psi_2}=\nabla_{\psi_1} \nabla_{\psi_2}$.
\end{defn}

The notion of an integrable iterative connection is motivated by the
correspondence to the integrable (ordinary) connections 
in characteristic zero (cf. Section \ref{char_zero}).

\begin{thm}
Let $\nabla$ be a higher connection on $M$. Then:
\begin{enumerate}
\item $\nabla$ is iterative if and only if for all $a,b\in K^{\sep}$:
$a.\difnabla\circ b.\difnabla=(a+b).\difnabla$ and if
and only if for all $a,b\in K^{\sep}$: $a.\difnabla\circ b.\nabla=(a+b).\nabla$.
\item If $\nabla$ is iterative, then for all iterative derivations
$\phi\in\ID_K(R)$ the $\phi$-derivation $\nabla_\phi$ is again
iterative. If $R$ has enough iterative derivations, \ie{} if for every nonzero $\omega\in \Dif_{R/K}$ there exists
an iterative derivation $\phi\in\ID_K(R)$ such that
$\tilde{\phi}(\omega)\ne 0$, then the converse
is also true (where $\tilde{\phi}:\Dif_{R/K}\to R[[T]]$ is the unique
homomorphism of cgas satisfying $\phi=\tilde{\phi}\circ \d$, cf. Thm. \ref{diff_exists}).
\end{enumerate}
\end{thm}

\begin{proof}
The first statement is seen by a calculation similar to the one in Lemma
\ref{it_eig}. 
For proving the second part, let $\phi\in\ID_K(R)$ and consider the
following diagram:

\centerline{\xymatrix@+20pt{
M \ar[r]^-{b.\nabla} \ar[d]_{b.\nabla_{\phi}} & {\Dif} \otimes M
\ar[r]^{a.\d_\Dif \otimes a.\nabla} \ar[d]_{\left((a.\phi)[[T]]\circ
\tilde{\phi}\right) \otimes a.\nabla_{\phi}}
\ar[dr]|{\left(\tilde{\phi}\circ a.\d_{\Dif}\right)\otimes
a.\nabla_{\phi}}
\ar@/^2pc/[rr]^{a.\difnabla} & {\Dif}\otimes_{a.\d(R)} ( {\Dif}
\otimes M) \ar[r]^-{\mu \otimes \id_M} \ar[d]^{\tilde{\phi}\otimes
\tilde{\phi}\otimes \id_M} &  {\Dif} \otimes M
\ar[d]^{\tilde{\phi}\otimes \id_M} \\
R[[T]]\otimes M \ar[r]^<(0.22){(a.\phi)[[T]]\otimes a.\nabla_{\!\phi}} 
\ar@/_2pc/[rrr]_{(a.\nabla_{\phi})[[T]]} & 
R[[T]]\otimes_{a.\phi(R)} M[[T]] \ar@{=}[r] &
R[[T]]\otimes_{a.\phi(R)} M[[T]] \ar[r]^-{\mu\otimes \id_M}
 &  R[[T]]\otimes M
}}

\noindent The square on the left commutes, since
$$b.\nabla_{\phi}=b.\left((\tilde{\phi}\otimes \id_M)\circ
\nabla\right) = (\tilde{\phi}\otimes \id_M)\circ (b.\nabla) .$$
The lower triangle commutes by Lemma \ref{it_eig}, since $\phi$
is iterative. The upper triangle commutes, since
$a.\nabla_{\phi}=(\tilde{\phi}\otimes \id_M)\circ (a.\nabla)$, and the
square on the right commutes, since $\tilde{\phi}$ is a
homomorphism of algebras. Furthermore the top of the diagram commutes
by definition of $a.\difnabla$ and the bottom commutes, since
$a.\nabla_\phi$ is a $(a.\phi)$-derivation.\\
So the whole diagram commutes and we obtain
$$(\tilde{\phi}\otimes \id_M)\circ (a.\difnabla)\circ (b.\nabla)
= (a.\nabla_\phi)[[T]]\circ (b.\nabla_\phi)=(a.\nabla_\phi)(b.\nabla_\phi)$$
for all iterative derivations $\phi\in\ID_K(R)$.\\
If $\nabla$ is iterative, we get 
$$(a+b).\nabla_\phi=(\tilde{\phi}\otimes \id_M)\circ (a+b).\nabla
=(\tilde{\phi}\otimes \id_M)\circ (a.\difnabla)\circ (b.\nabla)
=(a.\nabla_\phi)(b.\nabla_\phi)$$
by the first part of this theorem and so by Lemma \ref{it_eig_modul},
$\nabla_\phi$ is iterative.\\
In turn, from the commuting diagram we see that if $\nabla_\phi$ is
iterative for an iterative derivation $\phi\in\ID_K(R)$, we get
$$(\tilde{\phi}\otimes \id_M)\circ (a.\difnabla)\circ (b.\nabla)
=(\tilde{\phi}\otimes \id_M)\circ (a+b).\nabla$$
for this $\phi$. So if $R$ has enough iterative derivations and
$\nabla_\phi$ is iterative for all $\phi\in\ID_K(R)$ we obtain
$(a.\difnabla)\circ (b.\nabla)= (a+b).\nabla$, \ie{} $\nabla$ is
iterative.
\end{proof}

\section{The Tannakian Category of Modules with Iterative
  Connection}\label{categorial}

In this section we show -- assuming a slight restriction on the
ring $R$ -- that the finitely generated projective modules
(i.e. locally free of finite rank) with higher
connection form an abelian category and that the modules with
 iterative resp. integrable iterative
connection form full subcategories. Furthermore all these categories
form tensor categories over $K$ (in the sense of \cite{deligne2}); in
fact they even form Tannakian categories.

\begin{notation}
From now on let $R$ be a
regular ring over $K$ which is the localisation of a
finitely generated $K$-algebra, such that $K$ is algebraically closed
in $R$. Also keep in mind that we assumed $R$ to be an integral domain
and $K$ to be a perfect field.\\
Furthermore, in the following a pair $(M,\nabla)$ always denotes a
finitely generated $R$-module $M$ together with a higher connection
$\nabla:M\to \Dif \otimes_R M$, even if ``finitely generated'' is not
mentioned. We recall the fact given in Corollary
\ref{automatically_projective}, namely that such a module is always
projective. 
\end{notation}

\begin{defn}
Let $(M_1,\nabla_1)$ and $(M_2,\nabla_2)$ be $R$-modules with higher
connection. Then we call $f\in \Hom_R(M_1,M_2)$ a {\markdef morphism
  of modules with higher connection}, or a morphism for short, if
the diagram 

\centerline{
 \xymatrix{ M_1\ar[r]^f \ar[d]^{\nabla_1} & M_2 \ar[d]^{\nabla_2}\\
\Dif \otimes_R M_1 \ar[r]^{\id_{\Dif} \otimes f} & \Dif \otimes_R M_2}}
\noindent commutes.
The set of all morphisms $f\in \Hom_R(M_1,M_2)$ will be denoted by
$\Mor\bigl((M_1,\nabla_1), (M_2,\nabla_2)\bigr)$. If the higher
connections are clear from the context we will sometimes omit them.
\end{defn}

\begin{rem}
It is clear that the set of modules with higher connection together
with the sets of 
morphisms defined above forms a category. This category will be denoted
by  $\HCon(R/K)$. Furthermore the full subcategories of
modules with iterative connection
resp. integrable iterative connection will be denoted by
$\ICon(R/K)$ resp. $\IConint(R/K)$ and by definition we 
have a chain of inclusions
$$\HCon(R/K)\supset \ICon(R/K)\supset \IConint(R/K).$$
As the objects of $\HCon(R/K)$ are modules with 
an extra structure and the morphisms are special homomorphisms, we
have a forgetful functor $\forget:\HCon(R/K)\to \Mod(R)$, which is faithful.
\end{rem}

\begin{defn}\label{constructions}
Let  $(M_1,\nabla_1)$ and $(M_2,\nabla_2)$ be $R$-modules with higher
connection. Then we define a higher connection $\nabla_{\!\oplus}$ on
$(M_1\oplus M_2)$ by
$$\nabla_{\!\oplus}: M_1\oplus M_2 \xrightarrow{\nabla_1\oplus \nabla_2}
\Dif \otimes M_1 \oplus \Dif \otimes M_2 \xrightarrow{\isom} \Dif
\otimes (M_1\oplus M_2)$$
and a higher connection $\nabla_{\otimes}$ on $M_1\otimes_R M_2$ by
\begin{eqnarray*}\nabla_{\!\otimes}: M_1\otimes_R M_2
&\!\xrightarrow{\nabla_1\otimes \nabla_2}\!&
(\Dif \otimes_R M_1) \otimes_{d(R)} (\Dif \otimes_R M_2)
 \xrightarrow{\isom}\\
& \xrightarrow{\isom}& (\Dif \otimes_{d(R)} \Dif) \otimes_R
(M_1\otimes_R M_2) \xrightarrow{\mu\otimes \id
} \Dif \otimes_R (M_1\otimes_R M_2).
\end{eqnarray*}

Furthermore we define a higher connection $\nabla_H$ on $\Hom_R(M_1,M_2)$ by
the following:\\
For $f\in \Hom_R(M_1,M_2)$ the composition 
$$1\otimes M_1\hookrightarrow \Dif \otimes_R M_1
\xrightarrow{(\difnabla_1)^{-1}} \Dif \otimes_R M_1 
\xrightarrow{\id_{\Dif}\otimes f}  \Dif \otimes_R
M_2\xrightarrow{\difnabla_2} \Dif \otimes_R M_2$$
is an element of $\Hom_R(M_1,\Dif\otimes_R M_2)$, and we have a
natural isomorphism 
$\Hom_R(M_1,\Dif\otimes_R M_2) \isom \Dif\otimes_R \Hom_R(M_1, M_2)$.
In this sense we define 
$$\nabla_H(f):= \difnabla_2\circ (\id_{\Dif}\otimes f)\circ
\left(\difnabla_1\right)^{-1}|_{1\otimes M}.$$ 
\end{defn}

\begin{rem}
\begin{enumerate}
\item[(i)] If $\nabla_1$ is an iterative connection, the definition of $\nabla_H$
coincides with the definition given in \cite{roesch}, because then we
have $(\difnabla_1)^{-1}|_{1\otimes M_1}=-\difnabla_1|_{1\otimes
  M_1}=-\nabla_1$.
\item[(ii)] As the referee pointed out, in the definition of the
  tensor product it might be possible to replace the twisting
  isomorphism $(\Dif \otimes_R M_1) \otimes_{d(R)} (\Dif \otimes_R M_2)
 \isom (\Dif \otimes_{d(R)} \Dif) \otimes_R (M_1\otimes_R M_2)$ by a
 more general isomorphism in order to obtain a non-commutative
 theory. This might lead to an iterative variant of Y. Andr\'e's
 framework for the theory of differential and difference equations
 (cf. \cite{andre}) and also to a more abstract framework for the
 iterative $q$-difference theory of C. Hardouin (cf. \cite{hardouin}).
\end{enumerate}
\end{rem}

\begin{thm}
The category $\HCon(R/K)$ is an abelian category and the categories
$\ICon(R/K)$ and $\IConint(R/K)$ are abelian subcategories. 
\end{thm}

\begin{proof}
For all $(M_1,\nabla_1), (M_2,\nabla_2)\in \HCon(R/K)$ the set of morphisms
$\Mor(M_1,M_2)$ is a subgroup of $\Hom_R(M_1,M_2)$ and so it is an
abelian group. Since $\Mod(R)$ is an abelian category, it is
sufficient to show that kernels, direct sums and so on in the
category $\Mod(R)$ can be equipped with a higher connection
resp. (integrable) iterative connection and that all necessary
homomorphisms (like the inclusion map of the kernel into the module)
are morphisms.\\
The only point at which one has to be careful is the kernel: For a
morphism $f\in\Mor(M_1, M_2)$, the higher
connection $\nabla_1$ induces a map $\nabla_1|_{\Ker(f)}:\Ker(f)\to
\Ker(\id_{\Dif} \otimes f)$ and one has to verify that
$\Ker(\id_{\Dif} \otimes f)=\Dif\otimes \Ker(f)$. But this follows
from the fact that $\Dif_k$ is a projective $R$-module for all
$k\in\NN$ and hence $\Dif\otimes -$ is an exact functor.
(See \cite{roesch}, Ch.4.1, for details.) 
\end{proof}

Now we will show that these categories are tensor categories over
$K$ (we will use the definition of a tensor category over $K$ given in 
\cite{deligne2}). 
By the previous theorem,
they are all abelian, and by Corollary \ref{automatically_projective},
all modules that arise are projective and the category $\Projmod(R)$
of finitely generated projective $R$-modules is known to satisfy all
properties of a tensor category apart from being an abelian category.

So we define 
\begin{itemize}
\item the tensor product of $(M_1,\nabla_1)$ and $(M_2,\nabla_2)$ by
$$(M_1,\nabla_1)\otimes (M_2,\nabla_2):=(M_1\otimes_R M_2,\nabla_{\!\otimes})$$
(this tensor product is obviously associative and commutative),
\item the unital object $\uob:=(R,\d_R)$ \ ($R\otimes_R M\to M,r\otimes
m\mapsto rm$ is easily seen to be a morphism for all $(M,\nabla)\in
\HCon(R/K)$),
\item the dual object to $(M,\nabla)$ by
$$(M,\nabla)^{\dual}:=(M^\dual, \nabla^\dual),$$
where $\nabla^\dual(f):=\d_\Dif \circ (\id_\Dif \otimes f)\circ
(\difnabla^{-1})|_{1\otimes M}\in \Hom_R(M,\Dif)$\,
for $f\in M^\dual=\Hom_R(M,R)$\\ 
(here we used that $\Dif \otimes_R
R\isom \Dif$ and $\Hom_R(M,\Dif)\isom \Hom_R(M,R)\otimes \Dif$.
Cf. the definition of $\nabla_H$ in Definition \ref{constructions}), 
\item the internal hom object of $(M_1,\nabla_1)$ and $(M_2,\nabla_2)$
by
$$\iHom\left((M_1,\nabla_1),(M_2,\nabla_2)\right)
:=\left(\Hom_R(M_1,M_2),\nabla_H\right).$$
\end{itemize}
Furthermore we recognize that every endomorphism in $\End(\uob)$ is
given by the image of $1\in R$, which has to be constant. Since all
constants are algebraic over $K$ and $K$ is 
algebraically closed in $R$, $\End(\uob)$ is isomorphic to $K$.

\begin{lem}
For all $(M_1,\nabla_1),(M_2,\nabla_2)\in \HCon(R/K)$ the isomorphism
of $R$-modules 
$$\iota_{M_1,M_2}:M_1^\dual \otimes_R M_2\to \Hom_R(M_1,M_2),\quad
f\otimes m\mapsto \{ v\mapsto f(v)\cdot m\}$$
is a morphism (and therefore an isomorphism) in $\HCon(R/K)$.
\end{lem}

\begin{proof}
For all $f\otimes m\in M_1^\dual \otimes_R M_2$ and for all $v\in M_1$,
we have
\begin{eqnarray*}
\nabla_H(\iota_{M_1,M_2}(f\otimes m))(v)&=&
\Bigl(\difnabla_2\circ \bigl( \id_\Dif \otimes \iota_{M_1,M_2}(f\otimes
m)\bigr) \circ (\difnabla_1^{-1})\Bigr) (1\otimes v)\\
&=& \difnabla_2\bigl( (\id_\Dif \otimes f)(\difnabla_1^{-1}(1\otimes v))\cdot (1 \otimes
m)\bigr)\\
&=&  \d_\Dif\bigl( (\id_\Dif \otimes
f)(\difnabla_1^{-1}(1\otimes v)\bigr) \cdot \nabla_2(m)
\end{eqnarray*}
and
\begin{eqnarray*}
(\id_\Dif \otimes \iota_{M_1,M_2})(\nabla_{\!\otimes}(f\otimes m))(v) 
&=&
(\id_\Dif \otimes \iota_{M_1,M_2})\Bigl(\bigl(\d_\Dif \circ
(\id_\Dif\otimes f)\circ (\difnabla_1^{-1}|_{1\otimes M})\bigr)\otimes \nabla_2(m)\Bigr)(v)\\
&=& \bigl(\d_\Dif \circ (\id_\Dif\otimes f)\circ
(\difnabla_1^{-1})\bigr)(1\otimes v)\cdot \nabla_2(m).
\end{eqnarray*}
So $\nabla_H\circ \iota_{M_1,M_2}=(\id_\Dif \otimes
\iota_{M_1,M_2})\circ \nabla_{\!\otimes}$, \ie{} $\iota_{M_1,M_2}$ is
a morphism. 
\end{proof}

\begin{lem}
Let $(M,\nabla)\in \HCon(R/K)$, and let $\eps_M:M\otimes M^\dual\to R$
and $\delta_M:R\to M^\dual\otimes M$ be the evaluation and
coevaluation homomorphisms given in the 
definition of a tensor category, \ie{}, $\eps_M(m\otimes f)=f(m)$ and
$\delta_M(1)=\iota_{M,M}^{-1}(\id_M)$. Then $\eps_M$ and $\delta_M$
are morphisms in $\HCon(R/K)$.
\end{lem}

\begin{proof}
For $m\otimes f\in M\otimes M^\dual$, we have
\begin{eqnarray*}
&& (\id_\Dif\otimes \eps_M)(\nabla_{\!\otimes}(m\otimes f))\\
&&\quad =
(\id_\Dif\otimes \eps_M)\Bigl((\mu\otimes \id)\circ \bigl(\nabla(m)\otimes
(\d_\Dif\circ (\id_\Dif \otimes f)\circ \difnabla^{-1}|_{1\otimes
  M}\bigr)\Bigr)\\ 
&&\quad = (\mu\otimes \id)\Bigl( \bigl(\id_\Dif \otimes
\bigl(\d_\Dif \circ (\id_\Dif \otimes f)\circ \difnabla^{-1}|_{1\otimes
  M}\bigr)\bigr)
(\nabla(m))\Bigr)\\
&&\quad = \Bigl((\mu\otimes \id)\circ (\d_\Dif \otimes \d_\Dif)\circ
(\id_{\Dif\otimes \Dif} \otimes f)\circ (\d_\Dif^{-1} \otimes \difnabla^{-1}|_{1\otimes
  M}) \circ
\nabla\Bigr) (m)\\
&&\quad =\Bigl( \d_\Dif \circ (\id_\Dif \otimes f)\circ
(\difnabla^{-1}) \circ  \nabla\Bigr) (m)\\
&&\quad = \Bigl( \d_\Dif \circ (\id_\Dif \otimes f)\Bigr) (1\otimes
m)\\
&&\quad = \d_R(f(m)) =\d_R(\eps_M(m\otimes f))
\end{eqnarray*}
\noindent So $\eps_M$ is a morphism.

Since $\iota_{M,M}$ is an isomorphism in $\HCon(R/K)$, $\delta_M$ is a
morphism if and only if $\iota_{M,M}\circ\delta_M$ is a morphism.
Now
$$\nabla_H((\iota_{M,M}\circ\delta_M)(1))=\nabla_H(\id_M)=
\difnabla \circ (\id_\Dif \otimes \id_M)\circ
\difnabla^{-1}|_{1\otimes M}=
\difnabla \circ \difnabla^{-1}|_{1\otimes M}= 1\otimes \id_M$$
and
$$(\id_\Dif \otimes (\iota_{M,M}\circ\delta_M))(\d_R(1))= (\id_\Dif
\otimes (\iota_{M,M}\circ\delta_M))(1\otimes 1)= 1\otimes \id_M,$$
so $\delta_M$ is a morphism.
\end{proof}

\begin{thm}\label{all_tensor_cat}
$\HCon(R/K)$, $\ICon(R/K)$ and $\IConint(R/K)$ are tensor
categories over $K$.
\end{thm}

\begin{proof}
Since we have already shown that these categories are abelian, that
$\HCon(R/K)$ is 
equipped with an associative and commutative tensor product, and that
$\eps_M$ and $\delta_M$ are morphisms, we already know that
$\HCon(R/K)$ is a tensor category. Hence, it only remains to show that
the two subcategories are closed under tensor products and duals.

It is checked immediately that for $(M_1,\nabla_1),(M_2,\nabla_2)\in
\ICon(R/K)$ the higher connection $\nabla_{\!\otimes}$ on $M_1\otimes
M_2$ satisfies
$(a.\difnabla_{\!\otimes} )\circ
(b.\nabla_{\!\otimes} )=(a+b).\nabla_{\!\otimes}$ for all $a,b\in
K^{\sep}$. One also checks easily that the higher connection
$\nabla_{\! 1}^{\dual}$ on $M^\dual$ satisfies
$\left(a.\difnabla_{\! 1}^\dual \right) \circ
\left(b.\nabla_{\! 1}^\dual\right)= (a+b).\difnabla_{\! 1}^\dual$ for all
$a,b\in K^{\sep}$, if $\nabla_{\! 1}$ is iterative. Hence $\ICon(R/K)$ is a
tensor category.

Assuming that $(M_1,\nabla_{\! 1}),(M_2,\nabla_2)\in \IConint(R/K)$,
the integrability conditions for $\nabla_{\!\otimes}$ and
$\nabla_{\! 1}^{\dual}$ are obtained by a short calculation using the
following lemma.
\end{proof}

\begin{lem}\label{tensor_product_and_dual}
Let $(M_1,\nabla_1),(M_2,\nabla_2)\in \HCon(R/K)$ and let $\psi\in
\HD_K(R)$. Then we have:
\begin{enumerate}
\item $(\nabla_{\!\otimes})_\psi= (\mu\otimes \id)\circ \bigl((\nabla_{\! 1})_\psi \otimes
(\nabla_2)_\psi\bigr)$,
\item For all $f\in \Hom_R(M_1,M_2)$:
$$\bigl((\nabla_H)_\psi\bigr)(f)\ =\ \left(\nabla_2\right)_\psi [[T]]\circ
(\id_{R[[T]]}\otimes 
f)\circ \left(\left(\nabla_{\! 1}\right)_\psi [[T]]\right)^{-1}|_{1\otimes
  M}.$$
\end{enumerate}
\end{lem}

\begin{proof}
We have,
\begin{eqnarray*}
(\nabla_{\!\otimes})_\psi &=& (\tilde{\psi}\otimes \id)\circ (\mu \otimes
\id)\circ (\nabla_1\otimes \nabla_2) \\
&=& (\mu\otimes \id)\circ \bigl(  (\tilde{\psi}\otimes \id_{M_1})\otimes
(\tilde{\psi}\otimes \id_{M_2})\bigr) \circ  (\nabla_1\otimes \nabla_2) \\
&=& (\mu\otimes \id)\circ \bigl((\nabla_1)_\psi \otimes
(\nabla_2)_\psi\bigr),
\end{eqnarray*}
which shows the first part. For the proof of the second part, consider the
following diagram: 
\begin{center}$\xymatrix@C+10pt@R-2pt{
\Dif\otimes_R M_1 \ar[dd]_{\tilde{\psi}\otimes \id_{M_1}} & 
\Dif\otimes_R M_1 \ar[l]_{\difnabla_1}^{\isom} \ar[r]^{\id\otimes f} &
\Dif\otimes_R M_2 \ar[r]^{\difnabla_2}_{\isom}
 &\Dif\otimes_R M_2 \ar[dd]_{\tilde{\psi}\otimes \id_{M_2}}\\ 
& 1\otimes M_1 \ar[u] \ar[d] & & \\
 R[[T]]\otimes_R M_1 & R[[T]]\otimes_R M_1
 \ar[l]_{\left(\nabla_1\right)_\psi [[T]]}^{\isom} \ar[r]^{\id\otimes
   f} & R[[T]]\otimes_R M_2 \ar[r]^{\left(\nabla_2\right)_\psi
   [[T]]}_{\isom} & R[[T]]\otimes_R M_2 
}$\end{center}
It is sufficient to show that both maps from the upper left corner of
the diagram to the lower right corner are equal.
Both parts of the diagram (starting at $1\otimes M$ in the middle)
commute by definition of 
$\left(\nabla_1\right)_\psi$ resp. $\left(\nabla_2\right)_\psi$.
Furthermore, $\Dif\otimes M_1$ is generated as an $\Dif$-module by
elements in $1\otimes M_1$ and since $\difnabla_1$ is an isomorphism,
$\Dif\otimes M_1$ is also generated as an $\Dif$-module by elements of
the form $\difnabla_1(1\otimes m)$ for $m\in M_1$.
The equality of the maps then follows from the $\Dif$-linearity of the
upper row and the $R[[T]]$-linearity of the lower row.
\end{proof}

\begin{thm}\label{tannakian_categories}
The categories $\HCon(R/K)$, $\ICon(R/K)$ and $\IConint(R/K)$
are Tannakian categories with the forgetful functor $\forget:\HCon(R/K)\to
\Mod(R)$ (restricted to the respective category) as fibre functor.
If moreover $R$ has a $K$-rational point, \ie{}, there exists a maximal
ideal $\m\ideal R$ with $K\isom R/\m$, then these categories are
neutral Tannakian categories with fibre functor
$\forget_K:\HCon(R/K)\xrightarrow{\forget} \Mod(R)
\xrightarrow{\otimes_R R/\m} \Vect(K)$.
\end{thm}

\begin{proof}
By construction, the functor $\forget$ is a fibre functor and so the
tensor cate\-gories $\HCon(R/K)$, $\ICon(R/K)$ and $\IConint(R/K)$
are Tannakian cate\-gories. If $R$ has a $K$-rational point, by\linebreak
\cite{deligne2}.2.8, $\forget_K$ is a fibre functor. This proves the
second part.
\end{proof}

By Tannakian duality, every neutral Tannakian category $\T$ over $K$
with fibre functor $\forget_K$ is equivalent to the category of finite
dimensional representations of a group scheme defined over $K$, called
the Tannaka group scheme (or fundamental group scheme)
$\Pi(\T,\forget_K)$ of $\T$. Furthermore, this group scheme is the
projective limit of all quotients
$G_V:=\Pi(\tgener{V}),\forget_K\restr{\tgener{V}})$, where $V$ ranges
over all objects of $\T$ and $\tgener{V}$ denotes the 
smallest Tannakian subcategory of $\T$ that contains $V$.

Using the Picard-Vessiot theory in the second part of this paper, we
obtain the following proposition:

\begin{prop}
Let $R$ have a $K$-rational point and let $K$ be algebraically
closed. Then the Tannaka group schemes $\Pi(\ICon(R/K),\forget_K)$ and
$\Pi(\IConint(R/K),\forget_K)$ are reduced group schemes.
\end{prop}

\begin{proof}
If $\ch(K)=0$, then by general theory all group schemes are
reduced, and nothing has to be shown. So assume $\ch(K)=p>0$.

Using the equivalence of categories given in Thm. \ref{equiv_cat}  and
the identification in Rem. \ref{Fdivided}, we obtain from
\cite{santos}, 2.3.1 that the Tannaka group scheme associated to
$\IConint(R/K)$ is reduced, even a perfect group scheme. (The reducedness also
follows from the reducedness of $\Pi(\ICon(R/K),\forget_K)$, since
$\Pi(\IConint(R/K),\forget_K)$ is a quotient.)

Since $\Pi(\ICon(R/K),\forget_K)$ is the projective limit of all
$G_{(M,\nabla)}:=
\Pi(\tgener{(M,\nabla)},\forget_K\restr{\tgener{(M,\nabla)}})$, it
suffices to show that all 
$G_{(M,\nabla)}$ are reduced. Let $F$ be the quotient field of $R$ and
$E/F$ be a PPV-extension for $F\otimes_R M$ as defined in Section
\ref{pv-theory}, which exists since $K$ is algebraically closed
(cf. Remark \ref{rem_on_pv-rings}). Then $G_{(M,\nabla)}$ is
isomorphic to the Galois group scheme $\uGal(E/F)$
(cf. Rem. \ref{equivalence_of_galois_groups}). By Prop. \ref{R_l}, we have
$\Ker(\d_F^{(1)})=F^p$, and hence by Cor. \ref{no_inseparable_ext},
$\uGal(E/F)\isom G_{(M,\nabla)}$ is reduced.
\end{proof}

\begin{rem}
One might ask whether the inclusions in the chain of categories \linebreak
$\HCon(R/K)\supset \ICon(R/K)\supset
\IConint(R/K)$ are strict.\\
Clearly, $\HCon(R/K)\ne  \ICon(R/K)$, because if for example $M$
is a free
$R$-module of dimension $1$ with basis $b_1\in M$, every
$\omega=\sum_{j=0}^\infty \omega_j\in \Dif_{R/K}$
with  $\omega_0=1$ defines a higher connection
$\nabla:M\to\Dif_{R/K}\otimes_R M, b_1\mapsto \omega \otimes b_1$, but in
general this higher connection is not iterative, because if $\nabla$
is iterative, $\omega$ satisfies the condition
$$0=\left(- \difnabla\circ \nabla\right)^{(2)}(b_1)
=(2\omega_2-\omega_1^2+\d_\Dif^{(1)}(\omega_1))\otimes b_1.$$
(The only exception is the case when $R$ is algebraic over
$K$, because in this case $\Dif_{R/K}=R$ and hence all
categories above are equivalent to $\Mod(R)$).

The inclusion $\ICon(R/K)\supset \IConint(R/K)$ is also strict in
general, because in the next section we will see that in
characteristic zero, the category $\ICon(R/K)$ is equivalent to the
category of modules with (ordinary) connection over $R$ and
$\IConint(R/K)$ is equivalent to the category of modules with
integrable connection over $R$, and it is known that those two
categories are different if for example $R=K(t_1,t_2)$. However, it is
also known that every (ordinary) connection is integrable if $\ch(K)=0$
and $R$ is an algebraic function field in one variable over $K$. In
Section \ref{proj_system}, we will see that also $\ICon(R/K)=
\IConint(R/K)$, if $R$ is an algebraic function field in one
variable over $K$ and $\ch(K)=p$.
\end{rem}

\section{Characteristic Zero}\label{char_zero} 

For $\ch(K)=0$, in general one gets the usual constructions of
derivations, differentials and connections by restricting to the terms
of degree $1$. On the other hand these constructions can be uniquely
extended to iterative derivations and iterative
connections. Moreover integrable connections, i.\,e. connections which
preserve commutators of derivations, correspond to integrable
iterative connections. This will be proven in this section.\\
So throughout this section, $K$ is a field of characteristic zero and
$R$ is an integral domain which is a regular ring and the localisation
of a finitely generated $K$-algebra. Furthermore we assume that $R$
has a maximal ideal $\m\ideal R$, such that $R/\m$ is a finite
extension of $K$. 
$M$ denotes a finitely generated $R$-module.

\begin{prop}\label{der_itder} \
\begin{enumerate}
\item The map $$ \Der(R/K) \longrightarrow \ID_K(R),
\partial\mapsto \phi_{\partial},$$ given by
$$\phi_{\partial}(r):=\sum_{n=0}^\infty \frac{1}{n!}\partial^n(r)T^n$$
for all $r\in R$, 
is a bijection and the inverse map is given by $\phi\mapsto \phi^{(1)}$.\\
For a given derivation $\partial$ on $R$
and the corresponding iterative derivation $\phi_{\partial}$ the map 
\linebreak $\iterate: \Der_R(M) \to \ID(M,\phi_{\partial}),
\partial_M\mapsto \Phi_{\partial_M}$ given by
$$\Phi_{\partial_M}(m):=\sum_{n=0}^\infty \frac{1}{n!}\partial_M^n(m)T^n,$$
for all $m\in M$, 
is a bijection and the inverse map is given by $\Phi\mapsto \Phi^{(1)}$.
\item The $R$-module $(\Dif_{R/K})_1$ is canonically isomorphic to the
module of (ordinary) differentials $\Omega_{R/K}$ and $\d^{(1)}:R\to
(\Dif_{R/K})_1\isom \Omega_{R/K}$ is the universal derivation.
\item For every iterative connection $\nabla$ on $M$, the map
$\nabla^{(1)}:M\to (\Dif_{R/K})_1\otimes M\isom \Omega_{R/K}\otimes M$
is a connection on $M$ and every connection $\nabla^{(1)}$ on $M$
uniquely extends to an iterative connection on $M$. Furthermore,
$\nabla$ is an integrable iterative connection if and only if
$\nabla^{(1)}$ is an integrable connection.
\end{enumerate}
\end{prop}

\begin{proof}
{\rm i)} Let $\partial\in \Der(R/K)$ be a derivation. Then for all $i,j\in\NN$:
$\frac{1}{i!}\partial^{i}\circ \frac{1}{j!}\partial^{j}=\binom{i+j}{i}
\frac{1}{(i+j)!}\partial^{i+j}.$ So $\phi_{\partial}$ is an iterative
derivation. On the other hand, for every iterative derivation $\phi$,
one obtains $\phi^{(k)}=\frac{1}{k!}(\phi^{(1)})^{k}$ for all
$k\in\NN$ by applying the formula
$\phi^{(i)}=\frac{1}{i}\phi^{(1)}\circ \phi^{(i-1)}$
inductively. Finally by Remark \ref{highder_formel}, for all
$r,s\in R$ we have $\phi^{(1)}(rs)=r\phi^{(1)}(s)+\phi^{(1)}(r)s$, \ie{}
$\phi^{(1)}\in \Der(R/K)$.\\
The bijection $\iterate: \Der_R(M) \to
\ID(M,\phi_{\partial})$ is shown analogously.

\noindent {\rm ii)} The construction of $(\Dif_{R/K})_1$ given in the
proof of Theorem 
\ref{diff_exists} is the same as the usual construction of
$\Omega_{R/K}$ (e.g. in \cite{hartshorne}, Ch. II.8).

\noindent {\rm iii)}
The bijection of the iterative connections and
the ordinary connections is shown analogous to the first part. So
we only prove the equivalence of the integrability
conditions. 
Let $\partial_1,\partial_2\in \Der(R/K)$ be derivations,
let $\phi_i:=\phi_{\partial_i}$ ($i=1,2$) be the corresponding iterative
derivations, and let $\nabla$ be an iterative connection on
$M$. By an explicit calculation one gets
$$\left(\phi_1\phi_2\phi_1^{-1}\phi_2^{-1}\right)^{(1)}=0 \quad
\text{and} \quad
\left(\phi_1\phi_2\phi_1^{-1}\phi_2^{-1}\right)^{(2)}= \partial_1\circ
\partial_2- \partial_2\circ \partial_1=\left[ \partial_1,\partial_2\right].$$
From this and the iterativity condition of $\nabla$, one
computes that
$$\nabla_{\phi_1\phi_2\phi_1^{-1}\phi_2^{-1}}^{(2)}=
\left(\nabla^{(1)}\right)_{\left[ \partial_1,\partial_2\right]}.$$ 
(The
last expression means that we  apply the ordinary connection
$\nabla^{(1)}$ to the derivation $\left[
  \partial_1,\partial_2\right]$.) On the other hand, by quite the same
calculation as before one obtains 
$$\left(\nabla_{\phi_1}\nabla_{\phi_2}\nabla_{\phi_1}^{-1}
  \nabla_{\phi_2}^{-1}\right)^{(2)}=
\nabla_{\phi_1}^{(1)}\circ \nabla_{\phi_2}^{(1)} -
  \nabla_{\phi_2}^{(1)}\circ \nabla_{\phi_1}^{(1)} =
\left[
  \left(\nabla^{(1)}\right)_{\partial_1},
  \left(\nabla^{(1)}\right)_{\partial_2}\right].$$   
So if $\nabla$ is an integrable iterative connection, then
$\nabla_{\phi_1\phi_2\phi_1^{-1}\phi_2^{-1}}^{(2)}=
  \left(\nabla_{\phi_1}\nabla_{\phi_2}\nabla_{\phi_1}^{-1} 
  \nabla_{\phi_2}^{-1}\right)^{(2)}$ and hence $\nabla^{(1)}$ is an
  integrable connection.

For the converse, consider the complete local ring $\hat{R}_\m$. We
first note that for an arbitrary \Alg{R} $B$, every higher derivation
$\psi\in\HD_K(R,B)$ can be extended canonically to a higher derivation 
$\psi_e\in\HD_K(\hat{R}_\m,\hat{R}_\m \otimes B)$ in the following
way: Every homogeneous component $\psi^{(k)}$ ($k\in\NN$) can uniquely
be extended to the localisation $R_\m$ (see
Prop. \ref{unique_extension}). This extension is continuous with
respect to the $\m$-adic topology, since for all $i\in\NN$,
$\psi^{(k)}(\m^i)\subseteq \m^{i-k}(R_\m \otimes B_k)$. So
$\psi^{(k)}$ can be uniquely extended to a map
$\psi_e^{(k)}:\hat{R}_\m\to \hat{R}_\m \otimes B_k$, which is
continuous with respect to the $\m$-adic topology. One easily verifies
that indeed $\psi_e:=\sum_{k=0}^\infty \psi_e^{(k)}:\hat{R}_\m\to
\hat{R}_\m \otimes B$ is a higher derivation.\\
Since the extension is canonical, we obtain the identities
$(\id_{\hat{R}_\m}\otimes \widetilde{\psi})\circ \d_{R,e}=\psi_e$
($\d_{R,e}$ denotes the extension of the universal derivation $\d_R$)
and $(\psi_1\psi_2)_e=\psi_{1,e}\psi_{2,e}$ for all $\psi_1,\psi_2\in
\HD_K(R)$.\\ 
Now let $\nabla$ be an iterative connection such that $\nabla^{(1)}$
is an integrable connection. By \cite{katz2}, Prop. 8.9, the
$\hat{R}_\m$-module 
$\hat{R}_\m \otimes_R M$ is a trivial differential module,
i.\,e., there is an $\hat{R}_\m$-basis $\vect{b}=(b_1,\dots, b_n)$ of
$\hat{R}_\m \otimes_R M$ such that $\nabla^{(1)}(\vect{b})=0$, where
$\nabla^{(1)}$ is extended to $\hat{R}_\m \otimes_R M$ in the same
manner as the higher derivations. Since $\nabla$ is iterative, this
implies $\nabla^{(k)}(\vect{b})=0$ for all $k>0$.\\
Hence for all $\psi_1,\psi_2\in \HD_K(R)$ and all vectors $\vect{y}\in
\left.\hat{R}\right._{\!\m}^{n}$, s.t. $\sum y_ib_i \in M$, we have:
$$\nabla_{\psi_1\psi_2}(\sum y_ib_i)=
(\widetilde{\psi_1\psi_2}\otimes \id)\left(\sum
  \d_{R,e}(y_i)\nabla(b_i)\right)
= (\widetilde{\psi_1\psi_2}\otimes \id)\left(\sum
  \d_{R,e}(y_i)\cdot b_i\right)= \sum (\psi_1\psi_2)_e(y_i)\cdot
b_i,$$
and
$$\nabla_{\psi_1}\nabla_{\psi_2}\left(\sum y_ib_i\right)=
\sum \psi_{1,e}[[T]](\psi_{2,e}(y_i))\cdot b_i = \sum
(\psi_1\psi_2)_e(y_i)\cdot b_i.$$ 
Hence $\nabla_{\psi_1\psi_2}=\nabla_{\psi_1}\nabla_{\psi_2}$,
i.\,e., $\nabla$ is an integrable iterative connection.
\end{proof}

As a consequence of the previous proposition, we obtain

\begin{thm}
Under the assumptions given above, the category $\ICon(R/K)$
(resp. $\IConint(R/K)$) 
of finitely generated $R$-modules with iterative connection
(resp. integrable iterative connection) and the category 
of finitely generated $R$-modules with connection (resp. integrable
connection) are equivalent.
\end{thm}

We end this section with a comparison of integrable iterative
connections and stratifications (cf. \cite{bert_ogus}, Def. 2.10):
From the previous theorem and the fact that for a smooth ring in
characteristic zero a stratification is equivalent to an integrable
connection (cf. \cite{bert_ogus}, Thm 2.15 for a sketch of the proof),
we deduce 
the following corollary. In the next section, we will see that the
corollary also holds if the characteristic of $K$ is not zero
(cf. Cor. \ref{conn_strat_pos}).

\begin{cor}\label{conn_strat_zero}
Let $R$ be smooth over $K$, then the category $\IConint(R/K)$ is
equivalent to the category of stratified modules over $R$.
\end{cor}

\section{Positive Characteristic}\label{proj_system}

In this section, we consider the case that $K$ has positive
characteristic $p$. Contrary to characteristic zero, iterative
derivations and iterative connections are not longer determined by the
term of degree $1$. Moreover, not every
derivation $\partial\in\Der(R/K)$ can be extended to an iterative
derivation $\phi$ with $\phi^{(1)}=\partial$, because the conditions on
an iterative derivation imply $(\phi^{(1)})^p=p!\cdot \phi^{(p)}=0$,
\ie{}, at least $\partial$ has to be nilpotent.

But there are some other structural properties: The main result is
that every module with integrable 
iterative connection gives rise to a projective system and vice versa,
similar to the correspondence of projective systems and iterative
differential modules
over function fields given in \cite{mat_hart}, Ch.\,2.2. In fact, when
$R$ is an algebraic function field in one variable, the projective
systems defined here are equal to those defined by 
Matzat and van der Put, which shows that in this
case the categories $\ICon(R/K)$, $\IConint(R/K)$, ${\bf Proj}_R$ and
${\bf ID}_R$ are equivalent. (Here ${\bf Proj}_R$ denotes the category
of projective systems over $R$ and ${\bf ID}_R$ denotes the category
of ID-modules as given in \cite{mat_hart}).

So in this section, let $K$ be a perfect field of characteristic $p>0$ and
let $R$ be an integral domain which is a regular ring and the
localisation of a finitely generated $K$-algebra. Furthermore let
$t_1,\dots, t_m$ denote a separable transcendence basis for $R$, \ie{}
$\Quot(R)$ is a separable algebraic extension of the rational function field
$K(t_1,\dots, t_m)$.\footnote{This includes the case $m=0$, although
  in this case everything becomes trivial.}

$R$ has a natural sequence of $K$-subrings $(R_l)_{l\in\NN}$ given by
$R_l:=R^{p^l}$.
The following proposition gives a characterisation of
this sequence by the higher differential:

\begin{prop}\label{R_l} {\bf (Frobenius Compatibility)}
For all $l\in\NN$: 
$$R_l=\bigcap_{0<j<p^l} \Ker(\d_R^{(j)}).$$ 
\end{prop}

\begin{proof}
Since $\d_R$ is a homomorphism of algebras,
$\d_R(R_l)=\d_R(R^{p^l})\subset (\Dif_{R/K})^{p^l}$ and therefore
$\d_R^{(j)}(r)=0 \ (0<j<p^l)$ for all $r\in R_l$. The other inclusion is
shown inductively: The case $l=0$ is trivial. Now let $r\in R$ satisfy
$\d_R^{(j)}(r)=0$ for $0<j<p^l$. By induction hypothesis $r\in
R_{l-1}$. So there exists $s\in R$ with $s^{p^{l-1}}=r$. If $s\not\in
R^p$, then $s$ is a separating element of $R$ and we can find
se\-pa\-ra\-ting variables $s=s_1,s_2,\dots,s_m$ for $R$, \ie{}
$\Quot(R)/K(s_1,\dots,s_m)$ is a finite separable extension. By
applying Prop. \ref{isom_of_diff} and Theorem
\ref{dif_formula}(b), we see that $\d_R^{(1)}(s)\ne 0$. And so 
$$0\ne \left(\d_R^{(1)}(s)\right)^{p^{l-1}}=
\d_R^{(p^{l-1})}\left(s^{p^{l-1}}\right) = \d_R^{(p^{l-1})}(r),$$
which is a contradiction. So $s\in R^p$ and $r\in R_l$.
\end{proof}

In the case of $R$ being an algebraic function field in one variable,
it was shown by F. K. Schmidt (see \cite{hasse_schmidt}, Thm. 10 and
15) that for 
an iterative derivation $\phi\in\ID_K(R)$ satisfying $\phi^{(1)}\ne
0$, we have $R^{p^l}= \bigcap_{0<j<p^l} \Ker(\phi^{(j)})$.
So in this case we obtain the same sequence of subfields, when
``only'' looking at an iterative derivation instead of the universal
derivation.

\begin{defn}
A {\markdef Frobenius compatible projective system over $R$}
(or an {\markdef Fc-projective system over $R$} for short) is a sequence
$(M_l,\varphi_l)_{l\in\NN}$ with the following properties:
\begin{enumerate}
\item For all $l\in\NN$, $M_l$ is a finitely generated $R_l$-module.
\item $\varphi_l:M_{l+1}\hookrightarrow M_l$ is a monomorphism of
$R_{l+1}$-modules that extends to an isomorphism
$\id_{R_l} \otimes \varphi_l:R_l\otimes_{R_{l+1}}M_{l+1}\to M_l$.
\end{enumerate}
A {\markdef morphism $\alpha:(M_l,\varphi_l)\to
(M'_l,\varphi'_l)$ of Fc-projective systems over $R$ } is a sequence
$\alpha=(\alpha_l)_{l\in\NN}$ of homomorphisms of modules
$\alpha_l:M_l\to M'_l$ satisfying $\varphi'_l\circ
\alpha_{l+1}=\alpha_l\circ \varphi_l$.
\end{defn}

\begin{rem}\label{Fdivided}
An Fc-projective system $(M_l,\varphi_l)_{l\in\NN}$ over $R$ is nothing
else than a flat bundle on $\Spec(R)$ (cf. \cite{gieseker}, Def. 1.1)
resp. an F-divided sheaf on $\Spec(R)$ (cf. \cite{santos}, Def. 4),
if one identifies $R_l=R^{p^l}$ 
with $R$ via the Frobenius homomorphism ${\bf F}_l:R\to R_l,x\mapsto
x^{p^l}$. Then all $M_l$
can be regarded as $R$-modules and the maps 
$\id_{R_l} \otimes\varphi_l:R_l\otimes_{R_{l+1}}M_{l+1}\to M_l$ become
$R$-linear isomorphisms $\sigma_l:{\bf F}_1^*(M_{l+1})\to M_l$ from the
Frobenius pullback of $M_{l+1}$ to $M_l$,
i.\,e. $(M_l,\sigma_l)_{l\in\NN}$ is a flat bundle on $\Spec(R)$.
\end{rem}

\begin{prop}
Every Fc-projective system $(M_l,\varphi_l)_{l\in\NN}$ over $R$ defines
an integrable iterative connection $\nabla$ on $M:=M_0$ satisfying
$$\bigcap_{0<j<p^l} \Ker(\nabla^{(j)})=\left(\varphi_0\circ \dots \circ
\varphi_{l-1}\right)(M_l).$$
For a morphism $(\alpha_l)_{l\in\NN}:(M_l,\varphi_l)\to
(M'_l,\varphi'_l)$ of Fc-projective systems over $R$, the homomorphism
$\alpha_0:M=M_0\to M'=M'_0$ then is a morphism of modules 
with higher connection.
\end{prop}

\begin{proof} 
The proof is similar to the construction of a
stratification related to a flat bundle in the proof of
\cite{gieseker}, Thm. 1.3.

In order to define $\nabla^{(k)}$, choose $l\in\NN$ such that
$p^l>k$ and let $\chi_l:R\otimes_{R_l} M_l \xrightarrow{\isom} M$ be
the composition of the isomorphisms $\id_R\otimes
\varphi_j:R\otimes_{R_{j+1}} M_{j+1} \to R\otimes_{R_j} M_j$ ($0\leq
j<l$). Then we define $\nabla^{(k)}$ to be the composition:
$$\nabla^{(k)}: M \xrightarrow{\chi_l^{-1}} R\otimes_{R_l} M_l
\xrightarrow{\d_R^{(k)} \otimes \id_{M_l}} \left(\Dif_{R/K}\right)_k
\otimes_{R_l} M_l \xrightarrow{\id_{\Dif}\otimes \chi_l}
\left(\Dif_{R/K}\right)_k \otimes_{R} M.$$
This is well defined, because $\d_R^{(k)}$ is $R_l$-linear, and is
also independent of the chosen $l$. The definition also shows that the
necessary conditions for $\nabla$ being an integrable iterative
connection are fulfilled modulo degrees $\geq p^l$, since $\d_R$ is an
integrable iterative connection. As $l$ can be chosen arbitrarily
large, $\nabla$ fulfills all conditions for being an integrable
iterative connection.

It remains to show that $\bigcap_{0<j<p^l}
\Ker(\nabla^{(j)})=\left(\varphi_0\circ \dots \circ 
\varphi_{l-1}\right)(M_l)$. Since we have just constructed an
iterative connection on $M$, by Corollary
\ref{automatically_projective}, $M$ is projective and by the same
argument, all $M_l$ are projective.
Hence $\Ker(\d_R^{(j)}\otimes \id_{M_l})=\Ker(\d_R^{(j)})\otimes_{R_l}
M_l$ for all $j<p^l$ and so
$$\bigcap_{0<j<p^l} \Ker(\nabla^{(j)})
= \chi_l(R_l\otimes_{R_l} M_l)
=\left(\varphi_0\circ \dots \circ \varphi_{l-1}\right)(M_l).$$

Finally, let $(\alpha_l)_{l\in\NN}:(M_l,\varphi_l)\to
(M'_l,\varphi'_l)$ be a morphism of Fc-projective systems over $R$. We
have to show that $\nabla' \circ \alpha_0= (\id_\Dif \otimes \alpha_0)\circ
\nabla$, or equivalently, that for all $k\in\NN$ 
$$\nabla'^{(k)} \circ \alpha_0=(\id_\Dif \otimes
\alpha_0)\circ \nabla^{(k)}.$$
But the last condition is easily seen to hold by choosing $l\in\NN$
such that $p^l>k$ and by using the definition of the iterative
connections above.
\end{proof}

In what follows, we will show that the converse is also true, \ie{}
that a
module with integrable iterative connection gives rise to an Fc-projective
system over $R$. For this, we consider the quotient field
$F:=\Quot(R)$ of $R$ and a monomial ordering on
$\Dif_{F/K}=F[[\d^{(i)}t_j]]$, namely the lexicographical order, where
the variables are ordered by 
$\d^{(i_1)}t_{j_1}>\d^{(i_2)}t_{j_2}$ if $i_1>i_2$ or if $i_1=i_2$ and
$j_1>j_2$. The leading term of $\omega\in\Dif_{F/K}$ (if it exists) is
then denoted by $\LT(\omega)$.

\begin{lem}
Let $\omega\in \Dif_{F/K}$ be homogeneous of degree $p^l$ and $\omega\not\in
F\Dif_{F/K}^{p^l}$. Let $\d^{(i_0)}t_{j_0}$ be the highest variable with
the property that there exist $e_0\in\NN$, $p\nmid e_0$ and a
monomial $\omega'\in\Dif_{F/K}$ such that $(\d^{(i_0)}t_{j_0})^{e_0}\omega'$
is a monomial term of $\omega$. Let $e_0$ and $\omega'$ be chosen such
that  $(\d^{(i_0)}t_{j_0})^{e_0}\omega'$ is maximal among those
monomials. Then for every $k\leq p^l(p-1)$, we have:
$$\LT(\d_{\Dif}^{(k)}(\omega))\leq e_0 \d^{(i_0+p^l(p-1))}t_{j_0}\cdot
(\d^{(i_0)}t_{j_0})^{e_0-1}\omega',$$
with equality if and only if $k=p^l(p-1)$ and $i_0<p^l$.
\end{lem}

\begin{proof}
For $i\in\NN$, $j\in\{1,\dots, m\}$, $e\in\NN_+$ and $k\in\NN$, we
have 
$$\d_\Dif^{(k)}\left((\d^{(i)}t_j)^e\right)= \sum_{k_1+\dots +k_e=k}
\binom{i+k_1}{i}\cdots \binom{i+k_e}{i} \d^{(i+k_1)}t_j\cdots
\d^{(i+k_e)}t_j.$$ 
So
\begin{eqnarray*}
\LT\left(\d_\Dif^{(k)}\left((\d^{(i)}t_j)^e\right)\right)&=&
e\cdot \binom{i+k}{i} \d^{(i+k)}t_j (\d^{(i)}t_j)^{e-1} \quad \text{if
} e\binom{i+k}{i}\ne 0\in\FF_p\\
\d_\Dif^{(k)}\left((\d^{(i)}t_j)^e\right)&=& 0 \qquad \text{if } p\mid
e\text{ and } p\nmid k \quad \text{ and }\\
\d_\Dif^{(k)}\left((\d^{(i)}t_j)^e\right)&=& \left(
\d_\Dif^{(\frac{k}{p})}\left((\d^{(i)}t_j)^{\frac{e}{p}}\right)\right)^p
\quad  \text{if }  p\mid e\text{ and } p\mid k.
\end{eqnarray*}
So for $k\leq p^l(p-1)$, a variable $\d^{(i)}t_j\ne \d^{(i_0)}t_{j_0}$
occurring in $\omega$ gives a contribution to $\d_{\Dif}^{(k)}(\omega)$
of variables
\begin{enumerate}
\item[(i)] less than $\d^{(i_0+k)}t_{j_0}$ if it occurs in a power not
divided by $p$ and
\item[(ii)] less than $\d^{(i+\frac{k}{p})}t_{j}$ otherwise.
\end{enumerate}
In the second case, $i\leq p^{l-1}$, since $\omega$ is of
degree ${p^l}$, and
so $i+\frac{k}{p}\leq p^{l-1}+p^{l-1}(p-1)=p^l$. So
$\d^{(i+\frac{k}{p})}t_{j}< \d^{(i_0+p^l)}t_{j_0}$.
Therefore the highest variable that may occur is
$\d^{(i_0+k)}t_{j_0}$ (or $\d^{(i_0+p^l)}t_{j_0}$ if $k<p^l$) and 
$\d^{(i_0+p^l(p-1))}t_{j_0}$ occurs if and only if $k=p^l(p-1)$ and
$\binom{i_0+p^l(p-1)}{i_0}\ne 0\in\FF_p$, \ie{} $i_0\ne p^l$.\\
The highest corresponding monomial then is

\hfill $e_0 \d^{(i_0+p^l(p-1))}t_{j_0}\cdot
(\d^{(i_0)}t_{j_0})^{e_0-1}\omega'.$
\end{proof}

\begin{prop}\label{icon_int_to_proj}
Every $R$-module $M$ with integrable iterative connection $\nabla$
defines an Fc-projective system $(M_l,\varphi_l)$ over $R$, where
$M_l:=\bigcap\limits_{0<j<p^l} \Ker(\nabla^{(j)})$ and
$\varphi_l:M_{l+1}\to M_l$ is the inclusion map, and a morphism
$f:(M,\nabla)\to (M',\nabla')$ of modules with higher connection
defines a morphism $\alpha:(M_l,\varphi_l)\to
(M'_l,\varphi'_l)$ of Fc-projective systems over $R$ by
$\alpha_l:=f\restr{M_l}$. 
\end{prop}

\begin{proof}
Since $\nabla$ is an integrable iterative connection on $M$,
$\nabla^{(1)}$ is an integrable connection on $M$ (cf. proof of
Prop. \ref{der_itder},iii) ), and is of $p$-curvature zero.
Now let $M_1:=\Ker\left(\nabla^{(1)}\right)$ ( $=\bigcap_{0<j<p^1}
\Ker(\nabla^{(j)})$, since $\nabla$ is iterative). Then by Cartier's
Theorem on the $p$-curvature (cf. \cite{katz2}, Thm. 5.1), $M_1$ is an
$R_1$-module and 
$R \otimes_{R_1} M_1\to M$ is an isomorphism of $R$-modules.

Next, we will show that $\nabla(M_1)\subset (\Dif_{R/K})^p\otimes_{R_1}
M_1$. Since
$(\Dif_{R/K})^p$ is 
isomorphic to $\Dif_{R_1/K}$ as an algebra by
the map $\left(\d^{(i)}x\right)^p\mapsto \d^{(i)}(x^p)$, this means that
essentially $\nabla\restr{M_1}$ is an integrable iterative connection
on the $R_1$-module $M_1$. It then follows inductively that
$R_l\otimes_{R_{l+1}} M_{l+1}\xrightarrow{\isom} M_l$ and that, essentially,
$\nabla\restr{M_{l+1}}$ is an integrable iterative connection on the
$R_{l+1}$-module $M_{l+1}$. 

Since $M_1$ and $\Dif_{R/K}$ are locally free, and hence localisation
is injective, it suffices to show the statement for the quotient field
$F:=\Quot(R)$ of $R$. For simplicity, we again write $M$ and 
$M_1$ for what should be $F\otimes_R M$ and $F_1\otimes_{R_1} M_1$:

Since $\nabla$ is iterative, we only have to show that
$\nabla^{(p^l)}(M_1)\subset (\Dif_{F/K})^p\otimes_{F_1} M_1$ for all $l\geq
1$.
So fix an $F_1$-basis $\vect{b}=(b_1,\dots,b_n)$ of $M_1$ (written as
a row) and let 
$A_l\in \Mat_n(\Dif_{p^l})$ with $\nabla^{(p^l)}(\vect{b})=
\vect{b}A_l$.\footnote{For simplicity we use vector notations:
$\vect{b}A_l$ denotes the row vector with $j$-th component
$\sum_{i=1}^n (A_l)_{ij} b_i$, and
$\nabla$ and $\d_{\Dif}$ are always applied to the components of a
vector or a matrix. Also we abbreviate $\Dif_{F/K}$ by $\Dif$.} From
$0=\difnabla^{(p^l)}(\nabla^{(1)}(\vect{b}))=
\difnabla^{(1)}(\nabla^{(p^l)}(\vect{b}))=\vect{b}
\d_\Dif^{(1)}(A_l)$ we conclude $\d_\Dif^{(1)}(A_l)=0$.
Assume there is an entry $\omega\in\Dif_{p^l}$
$\subset
F[\d^{(i)}t_j\mid i=1,\dots,p^l, j=1,\dots, m]$ of $A_l$ with
$LT(\omega)=r\d^{(p^l)}t_j$ (for some 
$r\in F$ and $j\in\{1,\dots,
m\}$).
Since $\d_\Dif^{(1)}(r\d^{(p^l)}t_j)=\d^{(1)}(r)\d^{(p^l)}t_j+
r\d^{(p^l+1)}t_j$, and since for all other monomials of $\omega$, the image
under $\d_\Dif^{(1)}$ doesn't contain the variable $\d^{(p^l+1)}t_j$,
we obtain $\d_\Dif^{(1)}(\omega)\ne 0$, a contradiction.
So $\omega\in F[\d^{(i)}t_j\mid i=1,\dots,p^l-1, j=1,\dots, m]$.

Furthermore, since $\nabla$ is iterative,
$\difnabla^{(p^l(p-1))}\circ
\nabla^{(p^l)}=\binom{p^{l+1}}{p^l}\nabla^{(p^{l+1})}=0$, and therefore
$$0=\difnabla^{(p^l(p-1))}(\vect{b}A_l)=\vect{b}\cdot
\d_\Dif^{(p^l(p-1))}(A_l) 
+ \sum_{k=0}^{p^l(p-1)-1}
\nabla^{(p^l(p-1)-k)}(\vect{b})\cdot \d_\Dif^{(k)}(A_l).$$
If $A_l\not\in\Mat_n(F \Dif^p)$, then by the previous lemma,
$\d^{(p^l(p-1))}(A_l)$ has an entry with leading term $e_0
\d^{(i_0+p^l(p-1))}t_{j_0}
\left(\d^{(i_0)}t_{j_0}\right)^{e_0-1}\cdot \omega'$ for some
$\omega'\in \Dif$, $i_0\leq p^l$ and $j_0\in\{1,\dots, m\}$, and the
variables occurring in $\d_\Dif^{(k)}(A_l)$ ($k<p^l(p-1)-1$) are less
than $\d^{(i_0+p^l(p-1))}t_{j_0}$. Moreover, those occurring in
$\nabla^{(p^l(p-1)-k)}(\vect{b})$ are even less than or equal to
$\d^{(p^l(p-1))}t_m$. So we would have
$\difnabla^{(p^l(p-1))}(\vect{b}A_l)\ne 0$.
Therefore $A_l\in \Mat_n(F\Dif^p)$. 

At last, since $\d_\Dif^{(1)}(A_l)=0$, in fact $A_l\in
\Mat_n(\Dif^p)$, which completes the proof.
\end{proof}

\begin{thm}\label{equiv_cat}
The category ${\bf Proj}_R$ of Fc-projective systems over $R$ and the
category $\IConint(R/K)$ are equivalent. Furthermore,
if $R$ is an algebraic function field in one variable over $K$ and
$\phi\in\ID_K(R)$ with $\phi^{(1)}\ne 0$, then they are also
equivalent to the category 
${\bf ID}_R$ of iterative differential modules over $(R,\phi)$
(cf. \cite{mat_hart}, Ch. 2 and \cite{mat_put}, Ch. 2) and to
the category $\ICon(R/K)$.
\end{thm}

\begin{proof}
The first statement follows immediately from the previous two
propositions, since the given maps are functors that are inverses to
each other.
The proof of Proposition \ref{icon_int_to_proj} shows that the
integrability condition is not necessary when $R$ is an algebraic
function field in one variable. So $\ICon(R/K)$ is equivalent to ${\bf
  Proj}_R$ in this 
case. Furthermore, Matzat and van der Put showed in \cite{mat_hart},
Thm. 2.8, resp. \cite{mat_put}, Prop. 5.1,
that ${\bf ID}_R$ is also equivalent to  ${\bf Proj}_R$.
\end{proof}

\begin{cor}\label{conn_strat_pos}
If $K$ is algebraically closed and $R$ is smooth over $K$, then 
the category $\IConint(R/K)$ is equivalent to the category of
stratified modules over $R$.
\end{cor}

\begin{proof}
By the previous theorem, the category $\IConint(R/K)$ is equivalent to
the category ${\bf Proj}_R$ of Fc-projective systems over $R$. Furthermore
under the given assumptions, ${\bf Proj}_R$ is equivalent to the
category of stratified modules over $R$, by
\cite{gieseker}, Thm. 1.3. So the statement follows.
\end{proof}

In the previous section, we have seen that the same corollary holds
for $\ch(K)=0$ (cf. Cor. \ref{conn_strat_zero}).
However, there is still no proof of this equivalence that works in
arbitrary characteristic.
Furthermore, it is an open question whether stratifications and
integrable iterative connections are equivalent or even related, when
$R$ is not smooth over $K$.


\section{Higher Connections on Schemes}\label{diff_schemes}

Next, we outline a generalisation of modules with iterative
connections to modules over schemes.\\
Throughout this section, let $K$ be a perfect field, let
$X$ be a nonsingular, geometrically integral
$K$-scheme which is separated and of finite type over $K$, and let
${\mathcal O}_X$ denote the structure sheaf of $X$.

\begin{defn}
We define the {\markdef sheaf of higher differentials on $X$}, denoted
by  $\Dif_{X/K}$, to be
the sheaf associated to the presheaf given by
$$U\mapsto \Dif_{{\mathcal O}_X(U)/K}$$
for each  open subset $U\subseteq X$ and by the restriction maps
$$D(\rho_V^U):\Dif_{{\mathcal O}_X(U)/K}\to \Dif_{{\mathcal O}_X(V)/K}$$
for all open subsets $V\subseteq U\subseteq X$, as defined in
Proposition \ref{isom_of_diff}, where $\rho_V^U:{\mathcal O}_X(U)\to {\mathcal O}_X(V)$ is
the restriction map of ${\mathcal O}_X$.
\end{defn}

\begin{rem}
By Proposition \ref{isom_of_diff}, for all open subsets  $V\subseteq
U\subseteq X$, the diagram\\
\centerline{\xymatrix@+10pt{
{\mathcal O}_X(U) \ar[r]^{\d_{{\mathcal O}_X(U)}} \ar[d]_{\rho_V^U} & \Dif_{{\mathcal O}_X(U)/K}
\ar[d]_{D(\rho_V^U)} \\
{\mathcal O}_X(V) \ar[r]^{\d_{{\mathcal O}_X(V)}} & \Dif_{{\mathcal O}_X(V)/K}
}}
\noindent commutes and so the collection of maps $\d_{{\mathcal O}_X(U)}$ induces a
morphism of sheaves of $K$-algebras $\d_X:{\mathcal O}_X \to \Dif_{X/K}$.
\end{rem}

\begin{prop}
If $X$ is an affine scheme, then the presheaf $U\mapsto
\Dif_{{\mathcal O}_X(U)/K}$ is in fact a sheaf.
\end{prop}

\begin{proof}
The given presheaf is a sheaf if and only if for all open subsets
$U\subseteq X$ and all open coverings $\bigcup\limits_{i\in I} U_i=U$, the
sequence
$$0 \to \Dif_{{\mathcal O}_X(U)/K}\to \prod_{i\in I} \Dif_{{\mathcal O}_X(U_i)/K}
\to \prod_{i,j\in I} \Dif_{{\mathcal O}_X(U_i\cap U_j)/K}$$
is exact. Since this is a sequence of \cgas, it suffices to show that
the sequence is exact in each homogeneous component.\\
For every open subset $V\subseteq U$, ${\mathcal O}_X(V)$
is a localisation of ${\mathcal O}_X(U)$ and so by Proposition
\ref{isom_of_diff},
$\Dif_{{\mathcal O}_X(V)/K}\isom {\mathcal O}_X(V)\otimes \Dif_{{\mathcal O}_X(U)/K}$.
By Corollary \ref{dif_projective}, the homogeneous components
$(\Dif_{{\mathcal O}_X(U)/K})_k$ 
($k\in\NN$) are projective ${\mathcal O}_X(U)$-modules  and therefore tensoring with
$(\Dif_{{\mathcal O}_X(U)/K})_k$ is exact. 
So the sequence above is exact in each homogeneous component, if the
sequence 
$$0 \to {\mathcal O}_X(U)\to \prod_{i\in I}{\mathcal O}_X(U_i) \to \prod_{i,j\in I}{\mathcal O}_X(U_i\cap
U_j)$$
is exact. But this is true since ${\mathcal O}_X$ is itself a sheaf.
\end{proof}

As an immediate consequence of this proposition, we have the
following corollary: 

\begin{cor}
For every affine open subset $U\subseteq X$, we have
$\Dif_{X/K}(U)=\Dif_{{\mathcal O}_X(U)/K}$. 
\end{cor}

\begin{defn}
Let $M$ be a coherent ${\mathcal O}_X$-module.
A {\markdef higher connection on $M$} is a morphism of
sheaves $\nabla:M\to \Dif_{X/K}\otimes_{{\mathcal O}_X} M$ which locally (\ie{}
on affine open subsets) is a higher connection in the sense of Section
\ref{highder_modules}. The higher connection $\nabla$ is called
{\markdef iterative resp. integrable iterative} if $\nabla$ locally is an
iterative resp. integrable iterative connection.
\end{defn}

\begin{rem}
\begin{enumerate}
\item By Corollary \ref{automatically_projective}, every coherent
${\mathcal O}_X$-module $M$ that admits a higher connection $\nabla:M\to
\Dif_{X/K}\otimes_{{\mathcal O}_X} M$ is locally free and of finite rank.
\item Following the notion of modules with higher connection over
  rings, the categories of coherent
  ${\mathcal O}_X$-modules with higher connection, with iterative
connection and with integrable iterative connection will be denoted
  by $\HCon(X/K)$, $\ICon(X/K)$ resp.
$\IConint(X/K)$.
By standard methods of algebraic geometry, one obtains that again
$\HCon(X/K)$, $\ICon(X/K)$ and $\IConint(X/K)$ are tensor
categories over $K$ and that they are Tannakian categories. And if $X$ has a
$K$-rational point, they are in fact neutral Tannakian categories over $K$.
\end{enumerate}
\end{rem}

\section{Picard-Vessiot theory}\label{pv-theory}

In Section \ref{categorial}, we showed that the category of modules with higher
connection $\HCon(R/K)$ with fibre functor $\forget_K:\HCon(R/K)\to
\Vect(K)$ is a neutral Tannakian category over $K$ and
that $\ICon(R/K)$ and $\IConint(R/K)$ are Tannakian subcategories. By
Tannaka duality, this means that $\HCon(R/K)$ is equivalent 
to the category of finite dimensional representations of a certain
group scheme and that $\ICon(R/K)$ and $\IConint(R/K)$ are equivalent
to the category of finite dimensional representations of quotients of
this group scheme. (In positive characteristic,
the group scheme associated to $\IConint(R/K)$ is isomorphic to the
fundamental group scheme for F-divided sheaves
$\Pi^{Fdiv}(\Spec(R),\forget_K)$ in \cite{santos}, by
Thm. \ref{equiv_cat} and Rem. \ref{Fdivided}.)
Furthermore, for every module with higher (or iterative or integrable iterative) connection
$(M,\nabla)$, one obtains the Tannakian Galois group $G_{(M,\nabla)}$,
which is the group scheme corresponding to the smallest Tannakian
subcategory that contains $(M,\nabla)$.
In this section, we obtain these Galois group schemes for modules with
iterative connection from another point of view, namely as
automorphisms of solution rings (so called pseudo Picard-Vessiot rings,
or PPV-rings for short).
The fact that the automorphism group
scheme of a PPV-ring of $(M,\nabla)$ is isomorphic to the
Tannakian Galois group scheme $G_{(M,\nabla)}$ can be shown in the
same manner as in \cite{put_singer}, Thm. 2.33 for differential
modules, or as in \cite{papanikolas}, Sections 3.5~--~4.5, for
t-motives, and is sketched in Remark \ref{equivalence_of_galois_groups}
at the end of this section.

Some of the constructions and proofs given here will be quite similar
to those of T.~Dyckerhoff in \cite{dyckerhoff}, who used Galois group
schemes for obtaining a differential Galois theory in characteristic
zero over non algebraically closed fields of constants.
However, we have to deal with an additional phenomenon occurring
in positive characteristic, namely inseparability of the extensions
and nonreduced group schemes.

\smallskip

Since the Picard-Vessiot theory we provide does not only work for
modules with iterative connections, but for a large class of higher
derivations, we make the following definition.

\begin{defn}
Let $F$ be a $K$-algebra and let $\sdif$ be an \Alg{F}. A higher
derivation $\theta:F\to \sdif$ will be called {\markdef iterable} if
the following hold:
\begin{enumerate}
\item[(i)] For all $k\in\NN$ the homogeneous component $\sdif_k$ is
  generated by $\{\theta^{(k)}(r)\mid r\in F\}$.
\item[(ii)] $\theta$ can be extended to a continuous endomorphism
  $\theta_\sdif:\sdif\to \sdif$, satisfying the iteration rule
  $\theta_\sdif^{(i)}\circ
  \theta_\sdif^{(j)}=\binom{i+j}{i}\theta_\sdif^{(i+j)}$
  ($i,j\in\NN$), or equivalently satisfying $(a.\theta_\sdif)\circ
  (b.\theta_\sdif)=(a+b).\theta_\sdif$ for all $a,b\in K^{\sep}$.
\end{enumerate}
Let $\theta$ be iterable, let $M$ be an $F$-module and $\Theta$ a higher
$\theta$-derivation on $M$. As for higher connections we can define an
endomorphism $\Theta_\sdif$ on $\sdif\otimes_F M$ by
$$\Theta_\sdif(\omega \otimes x):= \theta_\sdif(\omega)\cdot
\Theta(x)$$ 
for all $\omega\in\sdif$ and $x\in M$. The $\theta$-derivation $\Theta$ 
is called {\markdef iterable} if $\Theta_\sdif$ satisfies the
iteration rule $\Theta_\sdif^{(i)}\circ
  \Theta_\sdif^{(j)}=\binom{i+j}{i}\Theta_\sdif^{(i+j)}$ ($i,j\in\NN$).
\end{defn}

\begin{exmp}
The universal derivation $\d_F:F\to \Dif_{F/K}$ is an iterable higher
derivation with extension $\d_\Dif$. Other examples are 
appropriate extensions of the universal derivation to extensions of
$F$ (e.g. to PPV-rings $R$ over $F$ for some iterable higher
differential equation; cf. Def. \ref{ppv-ring}), or the canonical
extension of $\d_{K(t)}$ to $K((t))$ (i.\,e. a higher derivation
$\theta:K((t))\to K((t))\otimes_{K(t)} \Dif_{K(t)/K}$).\\
Further examples are iterative
derivations $\phi:F\to F[[T]]$ with $\phi^{(1)}\ne 0$ (the additional
assumption is only necessary to fulfill condition (i)), and also
$m$-variate iterative derivations $\phi:F\to F[[T_1,\dots, T_m]]$
defined by F. Heiderich in his Diplomarbeit (cf. \cite{heiderich}). 

Let $\theta_F:F\to \sdif$ be an iterable higher derivation and let
$L/F$ be a finite field extension. If $L/F$ is separable, then
$\theta_F$ 
extends uniquely to a higher derivation $\theta_L:L\to L\otimes \sdif$
(cf. Prop. \ref{unique_extension} and Ex. \ref{formally_etale_ex}),
which therefore is also iterable. If $L/F$ is not separable, there may
not exists an extension of $\theta_F$ to $L$. However, if $\sdif$ has
no nilpotent
elements, there exists at most one extension, which then is
iterable. This relies on the fact that for some $k\geq 0$, $L^{p^k}$
lies in a separable algebraic extension $\tF$ of $F$, and hence
$\theta_L(s)^{p^k}= \theta_{\tF}(s^{p^k})$ determines $\theta_L(s)$
uniquely for all $s\in L$.
\end{exmp}

\begin{rem}
If a higher derivation $\theta:F\to \sdif$ is iterable, the extension
$\theta_\sdif$ is unique, since $\sdif_k$ is generated by
$\theta^{(k)}(F)$ for all $k$. Furthermore, the iteration rule implies
that $\theta_\sdif$ is an automorphism of $\sdif$.
\end{rem}

\bigskip

From now on, we fix an arbitrary field $K$, a field $F$ containing $K$, an
\Alg{F} $\sdif$ having no zero-divisors, and an iterable higher
derivation $\theta:F\to \sdif$, such that $K=\{ t\in F\mid
\theta(t)=t\}$.

We introduce some notation.
\begin{defn}
A {\markdef $\theta$-ring} is an $F$-algebra $R$ together with an
iterable higher derivation\linebreak $\theta_R:R\to R\otimes \sdif$ that
extends $\theta$. The pair $(R,\theta_R)$ is  called a {\markdef
  $\theta$-field}, if $R$ is a field. The set
$C_R:=\{ r \in R\mid \theta_R(r)=r\otimes 1\}$ is called the {\markdef ring of
  constants} of $(R,\theta_R)$. An ideal $I\ideal R$ is called a
 {\markdef $\theta$-ideal} if for all $k\in\NN$,
 $\theta_R^{(k)}(I)\subseteq I\otimes \sdif_k$;\, $R$ is 
 {\markdef $\theta$-simple} if $R$ has no proper nontrivial
 $\theta$-ideals. Localisations of $\theta$-rings are again $\theta$-rings by
 $\theta(\frac{r}{s}):=\theta(r)\theta(s)^{-1}$ (as for iterative
 derivations one easily shows that these extensions are again iterable).
 The tensor product $R\otimes_F \tR$ of two $\theta$-rings $R$ and
 $\tR$ is a $\theta$-ring by
$$\theta_{R\otimes \tR}^{(k)}(r\otimes \tilde{r}):=\sum_{i+j=k}
\theta_R^{(i)}(r)\cdot \theta_{\tR}^{(j)}(\tilde{r})\in R\otimes \tR\otimes
\sdif,$$
for all $k\geq 0$, $r\in R$ and $\tilde{r}\in \tR$.
A homomorphism of $\theta$-rings $f:R\to \tR$ is called a
{\markdef $\theta$-ho\-mo\-mor\-phism} if $\theta_{\tR}\circ f=(f\otimes
\id_\sdif)\circ \theta_R$. The set of all $\theta$-homomorphisms is
denoted by $\Hom^{\theta}(R,\tR)$. Furthermore for $\theta$-rings
$R\geq \tR\geq F$, the set of
$\theta$-automorphisms of $R$ that leave the elements of $\tR$ fixed,
is denoted by $\Aut^\theta(R/\tR)$.\\
Given a $\theta$-ring $R$ and a $K$-algebra $L$, the tensor product
$R\otimes_K L$ can be given the structure of a $\theta$-ring by 
$\theta_{R\otimes_K L}(r\otimes a)=\theta_{R}(r)\otimes a$ ($r\in
R$, $a\in L$). We say that $\theta_R$ is {\markdef extended 
  trivially} to $R\otimes_K L$.
\end{defn}

Let $A=\sum_{k=0}^\infty A_k\in \GL_n(\sdif)$ with
$A_0=\vect{1}_n$ (identity matrix) and for all $k,l\in\NN$, 
$\binom{k+l}{l}A_{k+l}=\sum_{i+j=l} \theta^{(i)}(A_{k})\cdot
A_j\in\Mat(n\times n),\sdif_{k+l}).$
Then an equation
\begin{equation*}\label{ihde}\theta(\vect{y})=A\vect{y},\quad
  \tag{*}\end{equation*} 
where $\vect{y}$ is a vector of indeterminates, is called an {\markdef
  iterable higher differential equation}.\\
Notice that the condition on the matrices $A_k$ is the same as to say
that for an $F$-vector space $M$ with basis
$\vect{b}=(b_1,\dots,b_n)$, the $\theta$-derivation 
$\Theta$ defined by $\Theta(\vect{b})=\vect{b}A^{-1}$ is iterable, and
a vector $\vect{x}\in F^n$ is a solution of the equation~\eqref{ihde},
if and only if $\vect{b}\vect{x}\in M$ 
is constant, i.\,e. satisfies $\Theta(\vect{b}\vect{x})=\vect{b}\vect{x}$.

 \begin{defn}\label{ppv-ring}
A $\theta$-ring $(R,\theta_R)$ is called a {\markdef pseudo
  Picard-Vessiot ring} (PPV-ring) for $\theta(\vect{y})=A\vect{y}$, if the
  following holds: 
\begin{enumerate}
\item $R$ is $\theta$-simple.
\item There is a fundamental solution matrix $Y\in\GL_n(R)$, \ie{} an
  invertible  matrix satisfying $\theta_R(Y)=AY$.
\item As an $F$-algebra, $R$ is generated by the coefficients of $Y$
  and by $\det(Y)^{-1}$.
\item $C_R=C_F=K$.
\end{enumerate}
The quotient field $E=\Quot(R)$ is called a {\markdef pseudo PV-field}
for the equation \eqref{ihde}.
\end{defn}

\begin{rem}\label{rem_on_pv-rings}
Analogous to \cite{mat_hart}, Prop. 3.2, resp. \cite{mat_put}, Lemma
3.2, one shows that $R$ is an 
integral domain, so the quotient field $E$ exists. Furthermore, for
every $\theta$-simple $\theta$-ring which is finitely generated as an
$F$-algebra, its constants are algebraic over $K$. Hence if $K$ is
algebraically closed,  a PPV-ring for the equation
\eqref{ihde} is given by $R=S/P$, where $S:=F[X_{ij},\det(X)^{-1}\mid
i,j=1,\dots,n]$
is a $\theta$-ring by $\theta_S(X):=AX$ and $P\ideal 
S$ is a maximal $\theta$-ideal. Hence in this case PPV-rings always
exist and -- by a similar proof as for \cite{mat_hart}, Thm. 3.4 --
are unique up to $\theta$-isomorphisms.
\end{rem}

For a PPV-ring $R/F$ we define the functor
$$\underline{\Aut}^{\theta}(R/F): (\cat{Algebras} / K) \to (\cat{Groups}),
L\mapsto \Aut^{\theta}(R_L/F_L)$$
where $F_L:=F\otimes_K L$, $R_L:=R\otimes_K L$ and 
$\theta$ resp. $\theta_R$ is extended trivially to $F_L$ resp. $R_L$.

We will show that the functor
$\underline{\Aut}^{\theta}(R/F)$ is representable by a $K$-algebra of
finite type and hence is an affine group scheme of finite type over
$K$.

\begin{lem}\label{ideal_bijection}
Let $R$ be a $\theta$-simple $\theta$-ring with $C_R=K$, let $L$ be
a finitely generated $K$-algebra and $R_L:=R\otimes_K L$, with
$\theta$-structure trivially extended from $R$.
Then there is a bijection
\begin{eqnarray*}
\I(L) &{\xymatrix{\ar@{<->}[r]&}} & \I^{\theta}(R_L) \\
I & {\xymatrix{\ar@{|->}[r]&}} & R_L (1\otimes_K I) = R\otimes_K I \\
 J\cap (1\otimes_K L) & {\xymatrix{\ar@{<-|}[r]&}}& J
\end{eqnarray*}
\noindent between the ideals of $L$ and the $\theta$-ideals of $R_L$.
\end{lem}

\begin{proof}
Obviously, both maps are well defined, so we only have to show that
they are inverses to each other.

\noindent (i) We need to show that for $I\in \I(L)$, we have
$(R\otimes_K I)\cap  (1\otimes_K L)=I$.\\
It is clear that $I$ is contained in the left side. For the other inclusion,
let $\{e_i\mid i\in \tilde{N}\}$ be a $K$-basis of $I$. Then
  $(R\otimes_K I)$ is a free $R$-module with the same basis and an element
  $f=\sum_{i\in \tilde{N}} r_i\otimes e_i\in (R\otimes_K I)$ is
  constant, if and only if all $r_i$ are constant, \ie{} if $f\in I$.

\noindent (ii) We need to show that for $J\in \I_{\theta}(R_L)$, we
have $R\otimes_K (J\cap (1\otimes_K L))=J$.\\
It is clear that the left side is contained in $J$, since $J$ is an
  ideal. For the other inclusion, let $\{e_i\mid i\in N\}$ be a
  $K$-basis of $L$, where $N$ denotes an index set. Then $\{e_i\mid
  i\in N\}$ also is a basis for the free $R$-module $R_L$.

For any subset $N_0\subseteq N$ and $i_0\in N_0$, let
  $\fA_{N_0,i_0}\ideal R$ denote the ideal of all $r\in R$ such that there
  exists an element $g=\sum_{j\in N_0} s_j\otimes e_j\in J$ with
  $s_{i_0}=r$. We will show that $\fA_{N_0,i_0}$ is a $\theta$-ideal
  of $R$ and so by $\theta$-simplicity of $R$ is equal to $(0)$ or to
  $R$:\\
Let $r\in \fA_{N_0,i_0}$, $g=\sum_{j\in N_0} s_j\otimes e_j\in J$ with
  $s_{i_0}=r$ and $k\in \NN$. We have to show that
  $\theta_R^{(k)}(r)\in \fA_{N_0,i_0}\otimes \sdif_k$. So let
  $\{\omega_\alpha\}$ be an $F$-basis of $\sdif_k$ and let
$g_\alpha\in R_L$ such that $\theta^{(k)}(g)=\sum_\alpha g_\alpha\otimes
  \omega_\alpha$. Since $J$ is a $\theta$-ideal, we have $g_\alpha\in
  J$.
On the other hand, let $\theta^{(k)}(s_j)=\sum_\alpha s_{\alpha,
  j}\otimes \omega_\alpha$ for some $s_{\alpha, j}\in R$, then
$$\theta^{(k)}(g)=\sum_{j\in N_0} \theta^{(k)}(s_j)\otimes e_j
=\sum_{j\in N_0} \sum_\alpha s_{\alpha, j}\otimes e_j \otimes \omega_\alpha.$$
 So $g_\alpha=\sum_{j\in N_0} s_{\alpha,j}\otimes e_j$ and therefore 
$s_{\alpha,i_0}\in \fA_{N_0,i_0}$. Hence,
$$\theta_R^{(k)}(r)=\theta_R^{(k)}(s_{i_0})=\sum_\alpha s_{\alpha,
  i_0}\otimes \omega_\alpha\in \fA_{N_0,i_0}\otimes \sdif_k.$$

Now, let $N_0\subset N$ be a subset, which is minimal for the
  property that $\fA_{N_0,i_0}\ne (0)$ for at least one index $i_0\in
  N_0$ (minimal in the lattice of subsets). 
So there exists $f=\sum_{j\in N_0} r_j\otimes e_j\in J$ with
$r_{i_0}=1$ and by minimality of $N_0$, for all $k>0$ we obtain
$\theta^{(k)}(f)=\sum_{\substack{j\in N_0 \\ j\ne i_0}} \theta^{(k)}
  (r_j)\otimes e_j=0$.
Hence $f\in J\cap (1\otimes_K L)$.\\
 Now let $g=\sum_{j\in N} s_j\otimes
  e_j\in J$ be an arbitrary element and denote by $N_1$ the set of
  indices $j$ where $s_j\ne 0$. By definition, for all $i\in N_1$,
  $\fA_{N_1,i}\ne (0)$. Hence there is $N_0\subseteq N_1$ minimal as
  above, $i_0\in N_0$ and $f=\sum_{j\in N_0} r_j\otimes e_j\in J\cap
  (1\otimes_K L)$ with $r_{i_0}=1$. By induction on the magnitude of
  $N_1$, we may assume that $g-s_{i_0}f\in R\otimes_K
  (J\cap (1\otimes_K L))\subset J$. So $g=(g-s_{i_0}f)+s_{i_0}f\in  R\otimes_K
  (J\cap (1\otimes_K L))$ and hence $R\otimes_K
  (J\cap (1\otimes_K L))=J$.
\end{proof}

\begin{prop}\label{central_iso}
Let $R$ be a PPV-ring for the equation \eqref{ihde} and let $T\geq F$ be a
$\theta$-simple 
$\theta$-ring with $C_T=K$ such that there exists a fundamental solution
matrix $Y\in \GL_n(T)$. Then there exists a finitely generated $K$-algebra
$U$ and a $T$-linear $\theta$-isomorphism
$$\gamma_T: T\otimes_F R \to T\otimes_K U,$$
where (again) the $\theta$-structure is extended trivially to
$T\otimes_K U$.\\
(Actually $U$ is isomorphic to the ring of constants of $T\otimes_F R$.)
\end{prop}

\begin{proof}
$R$ is obtained as a quotient of $F[X,X^{-1}]$ with
$\theta$-structure given by $\theta(X)=AX$ (for short we write
$F[X,X^{-1}]$ instead of $F[X_{ij},\det(X)^{-1}]$) by a maximal
$\theta$-ideal $P\ideal F[X,X^{-1}]$. Let $L:=K[Z,Z^{-1}]=K[\GL_n]$. We then
define a $T$-linear homomorphism
$$\gamma_T:T \otimes_F F[X,X^{-1}]\rightarrow T \otimes_K
K[Z,Z^{-1}]$$
by $X_{ij}\mapsto \sum_{k=1}^n Y_{ik}\otimes Z_{kj}$ (or $X\mapsto
Y\otimes Z$ for short).
$\gamma_T$ is indeed an isomorphism and -- if we extend the
$\theta$-structure trivially to $T\otimes_K K[Z,Z^{-1}]$ -- $\gamma_T$ is a
$\theta$-isomorphism.\\
By the previous lemma, the $\theta$-ideal $\gamma_T(T\otimes P)$ is
equal to $T\otimes I$ for an ideal $I\ideal K[Z,Z^{-1}]$. So for $U:=
K[Z,Z^{-1}]/I$, $\gamma_T$ induces a
$\theta$-isomorphism 
$$\gamma_T:T\otimes_F R\to T\otimes_K U.$$
\end{proof}

\begin{prop}\label{aut_is_affine}
Let $R$ be a PPV-ring over $F$. Then the group functor
$\underline{\Aut}^{\theta}(R/F)$ is represented by the finitely generated
$K$-algebra $U=C_{R\otimes_F R}$, \ie{} $\underline{\Aut}^{\theta}(R/F)$ is
an affine group scheme of finite type over $K$, which we call the
{\markdef Galois group scheme} $\underline{\Gal}(R/F)$ of $R$ over $F$
or also the Galois group scheme $\underline{\Gal}(E/F)$ of
$E:=\Quot(R)$ over $F$.
\end{prop}

\begin{proof}
First we show that for every $K$-algebra $L$ any $F_L$-linear
$\theta$-homomorphism $f:R_L\to R_L$ is an 
isomorphism: The kernel of such a homomorphism $f$ is a $\theta$-ideal of
$R_L$ and so by Lemma \ref{ideal_bijection}, it is generated by
constants, \ie{} elements in $1\otimes L$. But $f$ is $L$-linear and so
$\Ker(f)=\{0\}$. If $X\in \GL_n(R)$ is a fundamental matrix, then
$f(X)\in \GL_n(R_L)$ is also a fundamental matrix and so there is a
matrix $D\in \GL_n(C_{R_L})=\GL_n(L)$ such that $X=f(X)D=f(XD)$. Hence
$X_{ij},\det(X)^{-1}\in \Ima(f)$ ($i,j=1,\dots, n$), and since $R$ is
generated over $F$ by the $X_{ij}$ and by $\det(X)^{-1}$, the
homomorphism $f$ is also surjective. 

Using the isomorphism $\gamma:=\gamma_R$ of Prop. \ref{central_iso},
for a $K$-algebra $L$, we obtain a chain of isomorphisms:
\begin{eqnarray*}
\Aut^{\theta}(R_L/F_L) &=& \Hom_{F_L}^{\theta}(R_L,R_L)=
\Hom_{F_L}^{\theta}(F_L\otimes_F R,R_L)\\
&\isom & \Hom_{R}^{\theta}(R\otimes_{F} R,R_L) \\
&\isom & \Hom_{R}^{\theta}(R\otimes_K U,R_L) \\
&\isom & \Hom_{K}^{\theta}(U,R_L) \\
&\isom & \Hom_K(U,L)
\end{eqnarray*}
So $U$ is representing the functor $\underline{\Aut}^{\theta}(R/F)$.
\end{proof}

\begin{rem}\label{rho}
A careful look at the isomorphisms in the previous proof shows that
the universal object $\id_U\in\Hom_K(U,U)$ corresponds to the
$\theta$-automorphism $\rho\otimes \id_U:R\otimes_K U\to R\otimes_K U$,
where $\rho=\gamma_R\circ (1\otimes \id_R):R\to R\otimes_F R\to
R\otimes_K U$. Furthermore we obtain that the action of $g\in
\underline{\Aut}^{\theta}(R/F)(L)= \Hom_K(U,L)$ on $r\in R$ is given
by 
$$g.r=(\id_R\otimes g)\bigl(\gamma_R(1\otimes r)\bigr)\in R\otimes_K L.$$
\end{rem}

\begin{cor}\label{spec_is_torsor}
Let $R$ be a PPV-ring over $F$ and $\G:=\underline{\Gal}(R/F)$ the
Galois group scheme of $R$. Then $\Spec(R)$ is a $\G_F$-torsor.
\end{cor}

\begin{proof}
The isomorphism $\gamma=\gamma_R$ of Proposition \ref{central_iso}
determines an isomorphism of schemes
$$\Spec(\gamma): \Spec(R)\times_F \G_F=\Spec(R)\times_K
\G\to \Spec(R)\times_F \Spec(R).$$
By the previous remark and $R$-linearity of $\gamma$, the composition
of $\Spec(\gamma)$ with the second projection is the morphism which
describes the 
action of $\G_F$ on $\Spec(R)$, and the composition of $\Spec(\gamma)$
with the first projection is equal to the first projection
$\Spec(R)\times_F \G_F\to \Spec(R)$.
In other words, $\Spec(R)$ is a $\G_F$-torsor.
\end{proof}

The next proposition shows that being a torsor indicates a $\theta$-simple
$\theta$-ring to be a PPV-ring.

\begin{prop}\label{torsor_implies_ppv}
Let $R/F$ be a $\theta$-simple $\theta$-ring with constants
$C_R=K$. Further let 
$\G\leq \GL_{n,K}$ be an affine group scheme over $K$ and assume that
$\Spec(R)$ is a $\G_F$-torsor such that the corresponding isomorphism
$\gamma:R\otimes_F R\to R\otimes_K K[\G]$ is a $\theta$-isomorphism.
Then $R$ is a PPV-ring over $F$.
\end{prop}

\begin{proof}
Since $\Spec(R)$ is a $\G_F$-torsor, the fibration
$\Spec(R)\times^{\G_F} \GL_{n,F}$ is a $\GL_{n,F}$-torsor.
(The scheme $\Spec(R)\times^{\G_F} \GL_{n,F}$ is obtained as the
quotient of the 
direct product by the  action of $\G_F$ given by $(x,h).g:=(xg,g^{-1}h)$,
and is a right $\GL_{n,F}$-scheme by the action on the second factor.) By
Hilbert~90, every $\GL_{n,F}$-torsor is trivial, \ie{} we have a
$\GL_{n,F}$-equivariant isomorphism \linebreak
$\Spec(R)\times^{\G_F} \GL_{n,F} \to\GL_{n,F}$. Then the closed
embedding
$\Spec(R)\hookrightarrow \Spec(R)\times^{\G_F} \GL_{n,F} \to\GL_{n,F}$
induces an epimorphism $F[X,X^{-1}]\to R$ which is
$\G_F$-equivariant. Let the image of $X$ be denoted by $Y$. We then obtain
that the action of $\G$ on $Y$ is given by $Y\mapsto Y\cdot g$ for any
$L$-valued point $g\in\G(L)\subset \GL_{n}(L)$. Since by
assumption for every 
$K$-algebra $L$, the action of
$\G(L)$ commutes with $\theta$, the matrix $\theta(Y)Y^{-1}$ is
$\G$-invariant. So $\theta(Y)Y^{-1}=:A\in \GL_n(\sdif)$, and $Y$ is a
fundamental solution matrix for the equation
$\theta(\vect{y})=A\vect{y}$.
Hence $R$ is a PPV-ring.
\end{proof}

\begin{rem}\label{equivalence_of_galois_groups}
As indicated in the beginning of this section, the Tannakian Galois
group scheme
$G_{(M,\nabla)}$
  of a module with iterative connection and the 
Galois group scheme of a PPV-extension for $(M,\nabla)$ are
isomorphic. We now sketch this isomorphism.

So let $K$ be a perfect field, and let $S$ be a regular integral domain
which is the localisation of a finitely generated $K$-algebra, and such
that there is a maximal ideal $\m\ideal S$ with $S/\m\isom
K$. Furthermore, let $(M,\nabla)\in \ICon(S/K)$, and let $F:=\Quot(S)$
denote the quotient field of $S$ and $\theta:=\d_F:F\to \Dif_{F/K}$
the universal derivation of $F$. Since $M$ is a locally free module
(cf. Cor. \ref{automatically_projective}), by \cite{hartshorne},
Ch.II, Lemma 8.9, there exists a basis $\vect{b}:=(b_1,\dots, b_n)$ of the
$F$-vector space $F\otimes_S M$ with $b_i\in M$ ($i=1,\dots,n$), and such
that the residue classes in $M/\m M$ form a $K$-basis of $M/\m M$.
We assume that there exists a PPV-ring $R$ for the corresponding iterable
differential 
equation $\nabla(\vect{b}\vect{x})=\vect{b}\vect{x}$ with fundamental
solution matrix $Y\in\GL_n(R)$.

For obtaining the correspondence, we fix an isomorphism of
$R$-modules $\varphi:R\otimes_S M\to R\otimes_S M$ given by
$\varphi(\vect{b})=\vect{b}Y$. The correspondence is then given as
follows: 

For any $K$-algebra $L$ (with trivial $\theta$-structure),
an element $\sigma\in G_{(M,\nabla)}(L)$ is determined by $\sigma_M\in
\GL( L\otimes_K\forget_K(M))$ which can be identified with a matrix
$D_\sigma\in \GL_n(L)$ by  $\sigma_M(\vect{b})=\vect{b}D_\sigma$. So
$\sigma_M$ induces an $(R\otimes_K L)$-linear automorphism
$\tilde{\sigma}_M$ of $(R\otimes_K L)\otimes_S M$ by
$\vect{b}\mapsto \vect{b}D_\sigma$ and we obtain a
$\theta$-isomorphism $\hat{\sigma}:=\varphi\circ
\tilde{\sigma}_M\circ\varphi^{-1}$ of $(R\otimes_K L)\otimes_S M$
mapping the constant basis 
$\vect{b}Y$ to the constant basis $\vect{b}YD_\sigma$. One shows that 
this induces a $\theta$-isomorphism of $R\otimes_K L$ over $F\otimes_K
L$ given by $Y\mapsto YD_\sigma$, i.e. an element of $\uGal(R/F)(L)$.

On the other hand, every $\theta$-isomorphism of  $R\otimes_K L$ over $F\otimes_K
L$ is given by $Y\mapsto YD$ for some $D\in \GL_n(L)$ and by reversing
the steps above, one obtains an element $\sigma_M\in \GL(
L\otimes_K\forget_K(M))$, and one shows that indeed $\sigma_M$ defines
an element $\sigma\in G_{(M,\nabla)}(L)$.
\end{rem}

\section{Galois correspondence}\label{galois_corres}

In this section, we prove a Galois correspondence between all
intermediate $\theta$-fields of a PPV-extension $E/F$ and all closed
subgroup schemes of the Galois group scheme $\uGal(E/F)$. This
includes $\theta$-fields over which $E$ is inseparable and nonreduced
subgroup schemes, and hence is an improvement of the correspondence
given by Matzat and van der Put  (cf. \cite{mat_put}, Thm. 3.5), which
only considers reduced subgroup schemes and intermediate fields over
which $E$ is separable. (However, this separability condition is
missing in their statement, but has been added for example in
\cite{amano}, Thm. 2.5, and in \cite{heiderich}, Thm. 6.5.2.)

\begin{rem}\label{rel_to_takeuchi}
One should also mention the work of M. Takeuchi (cf. \cite{takeuchi})
on a Picard-Vessiot theory of 
so called C-ferential fields (a huge class of fields with extra
structure to which the iterative differential fields and the
$\theta$-fields defined below belong). But Takeuchi used a definition
of a PV-extension, that differs from ours and the ``usual'' one. The
main difference is that instead of requiring the existence of a
fundamental solution matrix he imposed a condition which is equivalent
to an isomorphism $R\otimes_F R\isom R\otimes_K C_{R\otimes_F
  R}$. (Here $F$ denotes a C-ferential field, $R$ a PV-ring over $F$,
$K=C_F=C_R$ the field of constants of $F$ and $R$, and $C_{R\otimes_F
  R}$ the constants of $R\otimes_F R$;
cf. \cite{takeuchi}, Def. 2.3).
Showing that this isomorphism also exists by our definition was the
statement of Proposition \ref{central_iso}. In fact, Proposition
\ref{central_iso} and 
Proposition \ref{torsor_implies_ppv} imply that both definitions
coincide in the case of $\theta$-fields.
Our Galois correspondence is quite the same as the one given by
Takeuchi (cf. \cite{takeuchi}, Thm 2.10), but we give maps in both
directions (Takeuchi only constructed the subgroup scheme
corresponding to an intermediate field) and also include the
correspondence of the separability condition and the reducedness
condition (separability and reducedness are not mentioned at all in
Takeuchi's work).
\end{rem}

In order to provide the Galois correspondence for PPV-extensions, we
need a functorial definition of invariants.
Let $S$ be a $K$-algebra and $\H/K$ be a subgroup functor of the
functor $\underline{\Aut}(S/K)$, \ie{} for every $K$-algebra $L$, the set
$\H(L)$ is a group acting on $S_L$ and this action is functorial in $L$. An
element $s\in S$ is then called {\markdef invariant} if for all $L$,
the element $s\otimes 1\in S_L$ is invariant under $\H(L)$. 
The ring of invariants is denoted by $S^{\H}$. (In \cite{jantzen},
I.2.10 the invariant elements are called ``fixed points''.) 
Let $E=\Quot(S)$ be the
localisation of $S$ by all non zero divisors. We call an element
$e=\frac{r}{s}\in E$ {\markdef invariant} under $\H$, if for all
$K$-algebras $L$ and all $h\in\H(L)$,
$$h.(r\otimes 1)\cdot (s\otimes 1)=(r\otimes 1)\cdot h.(s\otimes 1)\in
S\otimes_K L.$$
The ring of invariants of $E$ is denoted by $E^{\H}$.
One can easily verify that this definition of an invariant element
$e\in E$ is independent of the 
chosen representation $\frac{r}{s}$. 

\begin{rem}
One has to take care that in general the group functor $\underline{\Aut}(S/K)$ is
not a subgroup functor of $\underline{\Aut}(E/K)$, because not
every automorphism $S\otimes_K L\to S\otimes_K L$ can be extended to
an automorphism $E\otimes_K L\to E\otimes_K L$. Hence a subgroup
functor $\H$ of $\underline{\Aut}(S/K)$ does not have to be a subgroup
functor of $\underline{\Aut}(E/K)$. That is why we use this
more complicated definition of the invariants $E^{\H}$.
\end{rem}

In the following, let $R$ be a PPV-ring over $F$, $E=\Quot(R)$ its
quotient field and $\G=\underline{\Gal}(R/F)$ the Galois group
scheme of $R$ over $F$.

\begin{lem}
Let $\H\leq \G$ be a closed subgroup scheme and let
$\pi^{\G}_{\H}:K[\G]\to K[\H]$ denote the epimorphism corresponding to the
inclusion $\H\hookrightarrow \G$. Then an element $\frac{r}{s}\in E$
is invariant under the action of $\H$ if and only if $r\otimes
s-s\otimes r$ is in the kernel of the map
$$(\id_R\otimes \pi^{\G}_{\H})\circ \gamma:R\otimes_F R\to R\otimes_K K[\H].$$
\end{lem}

\begin{proof}
An element $\frac{r}{s}\in E$ is invariant under the action of $\H$ if
and only if it is invariant under the universal element in $\H$, namely
$\pi^{\G}_{\H}\in \G(K[\H])$. By Remark \ref{rho} and $R$-linearity of
$\gamma$, we have
$$(\id_R\otimes \pi^{\G}_{\H}) \bigl(\gamma (r\otimes s)\bigr)=
(r\otimes 1)\cdot \pi^{\G}_{\H}.(s\otimes 1)\in R\otimes_K K[\H].$$
Hence $r\otimes
s-s\otimes r$ is in the considered kernel if and only if $\frac{r}{s}$ is
invariant under $\H$.
\end{proof}

\begin{thm}\label{invariants}
For every closed subgroup scheme $\H\leq \G$, the ring $E^{\H}$ is a
$\theta$-field. Furthermore we have $E^{\H}=F$ if and only if $\H=\G$.
\end{thm}

\begin{proof}
By the previous lemma, it is obvious that $E^{\H}$ is a field. Next
let $\frac{r}{s}\in E^{\H}$. Then for all $k\in \NN$, we have
\begin{eqnarray*}
&&\theta^{(k)}(r\otimes s-s\otimes r)\cdot (s^k\otimes s^k)\\
&=&
\sum_{i_1+i_2+i_3=k}\theta^{(i_1)}(\frac{r}{s})s^k\theta^{(i_2)}(s)
  \otimes \theta^{(i_3)}(s) s^k - \theta^{(i_2)}(s)s^k\otimes
  \theta^{(i_1)}(\frac{r}{s}) s^k \theta^{(i_3)}(s) \\
&=& \sum_{i_1+i_2+i_3=k} \left( \theta^{(i_2)}(s)\otimes
  \theta^{(i_3)}(s)\right) \left(
  \theta^{(i_1)}(\frac{r}{s})s^k\otimes s^k - s^k\otimes
  \theta^{(i_1)}(\frac{r}{s})s^k \right) \\
&=& \sum_{i+j=k} \theta^{(i)}(s\otimes s)\left(
  \theta^{(j)}(\frac{r}{s})s^k\otimes s^k - s^k\otimes 
  \theta^{(j)}(\frac{r}{s})s^k \right).
\end{eqnarray*}
The left hand side lies in $\Ker((\id_R\otimes
\pi^{\G}_{\H})\circ \gamma)\otimes \sdif_k$, since the kernel is a $\theta$-ideal. So by
induction, we obtain
$(s\otimes s)\left(
  \theta^{(k)}(\frac{r}{s})s^k\otimes s^k - s^k\otimes 
  \theta^{(k)}(\frac{r}{s})s^k \right)\in \Ker((\id_R\otimes
\pi^{\G}_{\H})\circ \gamma)\otimes \sdif_k$ and hence
$\theta^{(k)}(\frac{r}{s})\in E^{\H}\otimes \sdif_k$.

For the second statement: If $\H=\G$, then $\pi^{\G}_{\H}=\id_{K[\G]}$
and the considered kernel is trivial. Hence $r\otimes s=s\otimes r\in
R\otimes_F R$ for all $\frac{r}{s}\in E^{\G}$. So $r=c\cdot s$ for an
appropriate element $c\in F$, \ie{} $\frac{r}{s}=c\in F$.\\
Assume $\H\lneq \G$. Since $\X=\Spec(R)$ is a $\G_F$-torsor, the
quotient scheme $\X/\G_F$ is equal to $\Spec(F)$, in particular it is
a scheme,
and since $\G_F$ and $\H_F$ are affine, $\G_F/\H_F$ also is a scheme. So by 
\cite{jantzen},I.5.16.(1), $\X/\H_F\isom \X\times^{\G_F} (\G_F/\H_F)$
is a scheme, too. Let
$\overline{\U}\subset \X/\H_F$ be an arbitrary affine open subset and
$\U=\pr^{-1}(\overline{\U}) \subset \X$ its inverse image, where
$\pr:\X\to \X/\H_F$ denotes the canonical projection. Then we get a
monomorphism $\pr_*:{\mathcal O}_{\X/\H_F}(\overline{\U})\to {\mathcal O}_{\X}(\U)$, whose
image is ${\mathcal O}_{\X}(\U)^{\H}$. If $E^{\H}=F$, then also
${\mathcal O}_{\X}(\U)^{\H}=F$. So for every open affine subset
$\overline{\U}\subset \X/\H_F$, we would have
${\mathcal O}_{\X/\H_F}(\overline{\U})=F$, \ie{} $\overline{\U}\isom \Spec(F)$ is a
single point. Hence $\X/\H_F=\Spec(F)$ which contradicts the
assumption $\H\lneq \G$.
\end{proof}


\begin{thm}{\bf (Galois correspondence)}\label{galois_correspondence}
\begin{enumerate}
\item There is an antiisomorphism of the lattices
$$\fH:=\{ \H \mid \H\leq\G \text{ closed subgroup schemes of }\G
\}$$
and
$$\fM:=\{ M \mid F\leq M\leq E \text{ intermediate $\theta$-fields} \}$$
given by 
$\Psi:\fH \to \fM,\H\mapsto E^{\H}$ and 
$\Phi:\fM \to \fH, M\mapsto \underline{\Gal}(RM /M)$.
\item\label{normal_subgroup} If $\H\leq \G$ is normal, then
  $E^{\H}=\Quot(R^{\H})$ and 
  $R^{\H}$ is a PPV-ring over $F$ with Galois group scheme 
$\underline{\Gal}(R^{\H}/F)\isom \G/\H$.
\item If $M\in\fM$ is stable under the action of $\G$, then $\H:=\Phi(M)$
 is a normal subgroup scheme of $\G$, $M$ is a PPV-extension of $F$ and
$\underline{\Gal}(R\cap M /F)\isom \G/\H$.
\item For $\H\in \fH$, the extension $E/E^{\H}$ is separable if and
  only if $\H$ is reduced.
\end{enumerate}
\end{thm}

\begin{proof}
{\rm i)} Let $M\in \fM$ be an intermediate $\theta$-field. Then the composite
$RM\subseteq E$ of $R$ and $M$ is a PPV-ring over $M$. Furthermore, 
the canonical $\theta$-epimorphism $RM\otimes_F R\to RM\otimes_M RM$ gives rise
to a $\theta$-epimorphism
$$RM\otimes_K K[\G]\xrightarrow{\gamma_{RM}^{-1}} RM\otimes_F R\to
RM\otimes_M RM.$$
By Lemma \ref{ideal_bijection}, the kernel of this epimorphism is
given by $RM\otimes_K I$ for an ideal $I\ideal K[\G]$. Let $\H$ denote
the closed subscheme of $\G$ defined by $I$, then $\gamma_{RM}$
induces an isomorphism
$$RM\otimes_M RM \xrightarrow{\isom} RM\otimes_K K[\H].$$
By construction, this isomorphism is the isomorphism $\gamma$ for the
base field $M$. Hence the subscheme $\H$ equals the Galois group scheme 
$\underline{\Gal}(RM/M)$. So $\underline{\Gal}(RM/M)$ is indeed a
closed subgroup scheme of $\G$.\\
From Theorem \ref{invariants} -- applied to the extension $E/M$ -- we
see that $E^{\underline{\Gal}(RM /M)}=M$, so $\Psi\circ
\Phi=\id_{\fM}$. On the other hand, for given $\H\in\fH$ and
$M:=E^{\H}$, we obtain a $\theta$-epimorphism $RM\otimes_M RM \to
RM\otimes_K K[\H]$ induced from $\gamma_{RM}$. This gives $\H$ as a
closed subgroup scheme of $\underline{\Gal}(RM/M)$. But
$(\Quot(RM))^{\H}=E^{\H}=M$, and so by Theorem \ref{invariants}, we
have $\H=\underline{\Gal}(RM/M)$. Hence $\Phi\circ
\Psi=\id_{\fH}$.

\noindent {\rm ii)} Let $\H\leq \G$ be normal. The isomorphism $\gamma$ is
  $\H$-equivariant (by the action of $\H$ on the right factor) and
hence we get a $\theta$-isomorphism 
$$R\otimes_F R^{\H}\isom R\otimes_K K[\G]^{\H}.$$
Since $\H$ is normal, $\G/\H$ is an affine group scheme with
$K[\G/\H]\isom K[\G]^{\H}$ (cf. \cite{demazure},III,\S 3, Thm. 5.6 and
5.8). Again
by taking invariants (this time $\H$ acting on the first factor) the
isomorphism above restricts to an isomorphism
$$R^{\H}\otimes_F R^{\H}\isom R^{\H}\otimes_K K[\G/\H].$$
$R^{\H}$ is $\theta$-simple, because for every $\theta$-ideal $P\ideal R^{\H}$,
the ideal $P\cdot R\ideal R$ is a $\theta$-ideal, hence equals $(0)$
or $R$, and so 
$P=(P\cdot R)^{\H}$ is $(0)$ or $R^{\H}$.
Since $F\leq R^{\H}\leq R$, we also have $C_{R^{\H}}=K$. 
So by Proposition \ref{torsor_implies_ppv}, $R^{\H}$ is a PPV-ring
over $F$ with Galois group scheme $\G/\H$.
It remains to show that $E^{\H}=\Quot(R^{\H})$:\\
Let $\tF:=\Quot(R^{\H})$ and $\tilde{\G}:=\uGal(E/\tF)$. Then $\H$ is
a normal subgroup of $\tilde{\G}$ and by the previous, $(R\cdot
\tF)^{\H}$ is a $\left(\tilde{\G}/\H\right)_{\tF}$-torsor. But
$(R\cdot\tF)^{\H}=R^{\H}\cdot \tF=\tF$, so $\tilde{\G}=\H$, and hence
$E^{\H}=E^{\tilde{\G}}=\tF=\Quot(R^{\H})$.

\noindent {\rm iii)} It suffices to show that $\H$ is normal in
$\G$. The rest then 
  follows from \ref{normal_subgroup}. Let $L$ be a $K$-algebra and let
  $h\in\H(L)$ and $g\in\G(L)$. Then for all $r\in R\cap M$, we have
$$ghg^{-1}.(r\otimes 1)=gh.(g^{-1}.(r\otimes 1))=g.(g^{-1}.(r\otimes
1))=(r\otimes 1),$$
since $g^{-1}.(r\otimes 1)\in (R\cap M)\otimes_K L$ by $\G$-stability of
$M$. So $ghg^{-1}\in\H(L)$, and therefore $\H$ is normal in $\G$.

\noindent {\rm iv)} Without loss of generality let ${\mathcal
  H}=\G$. If $\G$ is reduced, then 
$F=E^\G=E^{\G(\bar{K})}$ and hence $E\otimes_K \bar{K}/F\otimes_K
\bar{K}$ is separable, and so $E/F$ is separable.
On the other hand, if
$\G$ is not reduced, then $R\otimes_F R\cong R\otimes_K K[\G]$ is not
reduced. Hence $E\otimes_F E$ is not reduced. But this is just
one criterion for $E/F$ being inseparable (cf. \cite{matsumura},
beginning of Sect. 26).
%
\end{proof}

\begin{cor}\label{infinitesimal_group}
Let $E/F$ be a PPV-extension with Galois group scheme $\G$.
Then $E/F$ is a purely inseparable extension if and only if
$\G$ is an infinitesimal group scheme.
\end{cor}

\begin{proof}
Let $\G$ be infinitesimal and let $ev:K[\G]\to K$ denote the
evaluation map corresponding to the neutral element of the group. Then 
for any $\frac{r}{s}\in E$, we have
$(\id\otimes ev)(\gamma(r\otimes s-s\otimes r))=0.$
Since $\G$ is infinitesimal, the kernel of $ev$ is the
nilradical, and hence there is some $k\in\NN$ such that
$(r\otimes s-s\otimes r)^{p^k}=0$, where $p=\ch(F)$. Therefore
$r^{p^k}\otimes s^{p^k}=s^{p^k}\otimes r^{p^k} \in E\otimes_F E$
which means that $\frac{r^{p^k}}{s^{p^k}}\in F$. So $E/F$ is purely
inseparable.
On the other hand, if $E/F$ is purely inseparable, then 
$\G(\bar{K})=\Aut^{\theta}(E\otimes_K\bar{K}/F\otimes_K
\bar{K})$ is the trivial group, since
$E\otimes_K\bar{K}/F\otimes_K\bar{K}$ also is a purely 
inseparable extension. Hence $\G$ is infinitesimal.
\end{proof}

\begin{cor}\label{no_inseparable_ext}
Let $p:=\ch(F)>0$. If $\Ker(\theta_F^{(1)})=F^p$, then all
PPV-extensions $E/F$ are separable, and the corresponding Galois group
schemes are reduced.
\end{cor}

\begin{proof}
By Thm. \ref{galois_correspondence},iv), the separability of a
PPV-extension $E/F$ is equivalent to the reducedness of $\uGal(E/F)$.
Assume, there exists an inseparable PPV-extension $E/F$. 
Then there is an inseparable element $e\in E$, with
minimal polynomial $f(X)=\sum_{i=0}^n a_i X^{ip}\in F[X]$ 
for some $a_i\in F$, $a_n=1$. So we obtain:
$$0=\theta_E^{(1)}(f(e))= \sum_{i=0}^n \theta_F^{(1)}(a_i) e^{ip}+
\sum_{i=0}^n a_i (ip)e^{ip-1}\theta_E^{(1)}(e)= \sum_{i=0}^{n-1}
\theta_F^{(1)}(a_i) e^{ip}.$$
Since $f$ is the minimal polynomial of $e$, this implies
$\theta_F^{(1)}(a_i)=0$ for all $i$. Hence by assumption $a_i\in
F^p$. But then $f(X)=g(X)^p$, where $g(X)=\sum_{i=0}^n \sqrt[p]{a_i}
X^i\in F[X]$, which means that $f$ is not irreducible -- a contradiction.
\end{proof}

\begin{exmp}\label{exmp_group_schemes}
We consider some examples for subfields of
  $(K((t)),\theta)$, where $\theta:=\phi_t\in \ID_K(K((t)))$ is the
  iterative derivation with respect to $t$ (cf. Example
  \ref{phi_t_j}) and $K$ denotes a field of characteristic $p>0$.
  For simplicity we assume that $K$ is algebraically closed.
\begin{enumerate}
\item Let $F=K(t)\subseteq K((t))$. Then the IDE given by
$$\theta^{(p^l)}(y)=a_l t^{-p^l}y\quad
(a_l\in K)$$
has a PPV-ring $R=F[s,s^{-1}]$ and PPV-field $E=F(s)$, where $s$ is a
solution of the IDE. This implies that $\uGal(E/F)$ is a subgroup of
$\GG_m$. If the $a_l$ are chosen appropriately then we have 
$\uGal(E/F)=\GG_m$ (cf. \cite{mat_hart}, Thm. 3.13,
resp. \cite{mat_put}, Section 4) and $s$ is
transcendental over $F$. Furthermore the isomorphism
$\gamma:R\otimes_F R \to R\otimes_K K[\GG_m]$ is given by $1\otimes
s\mapsto s\otimes x$ ($K[\GG_m]=:K[x,x^{-1}]$).\\
All closed subgroup schemes of $\GG_m$ are given by the ideals
$(x^k-1)\ideal K[x,x^{-1}]$ for $k\in\NN$ (the so called subgroups
$\mu_k$ of $k$-th roots of unity) and the corresponding intermediate
$\theta$-fields are $E^{\mu_k}=\Quot(R^{\mu_k})=F(s^k)$.
Hence, there are also intermediate $\theta$-fields over which $E$ is
inseparable, namely for all $k>0$ that are divisible by $p$.
\item Let $F\subseteq K((t))$ be the subfield generated over $K$ by 
$t$, $s_1:=\prod_{l=0}^\infty \left(1+t^{a_lp^l}\right)$ and
$s_2:=\prod_{l=0}^\infty \left(1+t^{b_lp^l}\right)$, where $a_l,b_l\in
\{0,1,\dots, p-1\}$ are chosen  such that $t$, $s_1$ and $s_2$ are
algebraically independent. Consider the IDE
$$\theta^{(p^l)}\binom{y_1}{y_2}=
\begin{pmatrix} a_{l+1}\left(1+t^{a_{l+1}p^{l+1}}\right)^{-1} & 0\\
0 & b_{l+1}\left(1+t^{b_{l+1}p^{l+1}}\right)^{-1}\end{pmatrix}
\binom{y_1}{y_2} \quad (l\in\NN).$$
A solution of this IDE is given by $\binom{r_1}{r_2}\in K((t))^2$ with
$r_1^p=(1+t^{a_0})^{-1}\cdot s_1$ and $r_2^p=(1+t^{b_0})^{-1}\cdot
s_2$. Hence the corresponding PPV-ring
is $R=F[r_1,r_2]$ and the Galois group scheme -- a priori a subgroup
of $\GG_m\times \GG_m$ -- is equal to $\mu_p\times \mu_p$.
The action of the Galois group scheme on $R$ is given by the
homomorphism $\rho:R\to R\otimes_K K[\mu_p\times \mu_p]\isom
R\otimes_K K[x_1,x_2]/(x_1^p-1,x_2^p-1)$, which maps $r_i$ to
$r_i\otimes x_i$ ($i=1,2$).
Since the nontrivial subgroups of $\mu_p\times \mu_p$ are given by the
ideals $(x_1^{k}x_2-1)\ideal K[x_1,x_2]/(x_1^p-1,x_2^p-1)$
($k\in \{0,1,\dots, p-1\}$) and $(x_1-1)\ideal
K[x_1,x_2]/(x_1^p-1,x_2^p-1)$, there are exactly $p+1$ intermediate
$\theta$-fields unequal to $E$ and $F$, namely $F(r_1^kr_2)$
resp. $F(r_1)$.

So in this case, although $E/F$ has infinitely many intermediate fields,
there are only finitely many intermediate $\theta$-fields. 
\item Let $F\subseteq K((t))$ be the subfield generated over $K$ by 
$t,s_1:=\sum_{l=0}^\infty a_l t^{p^l}$ and $s_2:=\sum_{l=0}^\infty b_l
t^{p^l}$, where $a_l,b_l\in\FF_p$ are chosen such that $t,s_1,s_2$ are
algebraically independent. In this case we also have a purely
inseparable PPV-extension of degree $p^2$, namely
$E=F(r_1,r_2)\subseteq K((t))$ with $r_1^p=s_1-a_0t$,
$r_2^p=s_2-b_0t$. $r_1$ is a solution of the IDE
$$\theta^{(p^l)}\begin{pmatrix}1& r_1\\ 0 & 1\end{pmatrix}=
\begin{pmatrix}0& a_{l+1}\\ 0 & 0\end{pmatrix}\begin{pmatrix}1& r_1\\
  0 & 1\end{pmatrix} \quad (l\in\NN),$$
and $r_2$ a solution of the IDE with $a_{l+1}$ replaced by $b_{l+1}$.
Hence the Galois group scheme -- a subgroup
scheme of $\GG_a\times \GG_a$ -- is equal to $\alpha_p\times \alpha_p$
(where $\alpha_p$ denotes the Frobenius kernel inside $\GG_a$).\\
In this case, there are infinitely many intermediate $\theta$-fields,
  since $\alpha_p\times \alpha_p$ has infinitely many subgroups which
  are given by the ideals $(ay_1+by_2)\ideal
  K[y_1,y_2]/(y_1^p,y_2^p):=K[\alpha_p\times \alpha_p]$ ($a,b\in K$).\\
The action is given by $\rho:R\to R\otimes_K K[y_1,y_2]/(y_1^p,y_2^p)$
with $\rho(r_i)=r_i\otimes 1+1\otimes y_i$ ($i=1,2$).\\
So the corresponding intermediate $\theta$-fields are $F(ar_1+br_2)$,
$a,b\in K$. 

Comparing this example with the one before, we see that -- even for
finite extensions -- the Galois group scheme depends on the iterative
derivation. This is contrary to finite separable PPV-extensions, where the
Galois group is already determined by the extension of fields itself
(cf. \cite{mat_hart}, Thm. 1.15).
\end{enumerate}
\end{exmp}

\section{Finite inseparable extensions}\label{finite_inseparable}

In this section we compare our results for finite purely inseparable
PPV-extensions with the Galois theory for purely inseparable field
extensions given by Chase in \cite{chase}.

So let us first give a brief overview on some results in \cite{chase}:
Let $E/F$ be a purely inseparable field extension. Then the group
functor
$$G_t(E/F): (\cat{TruncAlg}/F)\to (\cat{Groups}), L\mapsto
\Aut(E\otimes_F L/L)$$ 
from the category of truncated $F$-algebras (\ie{} algebras of the
form\linebreak 
$F[t_1,\dots, t_r]/(t_1^{n_1},\dots, t_r^{n_r})$) to the category of
groups is representable by a truncated $F$-algebra $U$.
If the extension $E/F$ is modular (\ie{} for all $i\in\NN$, $E^{p^i}$
and $F$ are linearly disjoint over $E^{p^i}\cap F$), then
$E^{G_t(E/F)}=F$ and $\dim_F(U)=[E:F]^{[E:F]}$. 
In this case, there
is a Galois correspondence between the intermediate fields $F\leq
M\leq E$, s.\,t. $E/M$ is modular and certain closed subgroup schemes of
$G_t(E/F)$, given in the usual way by taking fixed fields respectively
subgroups fixing the given intermediate field. 
Furthermore, he showed that a purely inseparable field extension $E/F$
is modular if and only if there exists a truncated group scheme $\G$
(\ie{} an affine group scheme represented by a truncated $F$-algebra)
which acts on $E/F$, s.\,t. $\Spec(E)$ is a $\G$-torsor.
Given such a group scheme $\G$, then
$G_t(E/F)\isom \G(E\otimes_F -)$ as truncated group schemes over $F$.
However, although the group scheme $G_t(E/F)$ is unique, there might
be several such group schemes $\G$. 

Return now to the case that $E/F$ is a purely inseparable
PPV-extension and $\G:=\uGal(E/F)$. By
Prop. \ref{infinitesimal_group}, $\G$ is infinitesimal and since $K$
is perfect, $K[\G]$ is a truncated $K$-algebra (cf. \cite{demazure},
III, \S 3, Cor. 6.3) and so $F[\G]$ is a truncated $F$-algebra. As
shown in Corollary \ref{spec_is_torsor}, $E$ is a $\G_F$-torsor.\\
By the statements above, we obtain that $E/F$ is a modular field
extension and  that $G_t(E/F)$ equals $\G_F(E\otimes_F
-)$. So we can regain the truncated Galois group scheme $G_t(E/F)$
from our Galois group scheme $\uGal(E/F)$.

However, starting with $G_t(E/F)$, the iterable higher derivation
leads to a natural choice for a group scheme $\G_F\leq G_t(E/F)$ over
which $E$ is a torsor (namely $\uGal(E/F)_F$) and also gives a natural
description of the intermediate fields corresponding to the closed
subgroup schemes of $\G$. For instance, in Example
\ref{exmp_group_schemes},ii)+ iii), $F=K(t,s_1,s_2)$ is the rational
function field in three variables and $E/F$ is a purely inseparable
field extension of degree $p^2$ and exponent $1$. Hence in both
examples, we have the same (abstract) field extension. But in one case the
iterable higher derivation leads to the Galois group scheme
$\uGal(E/F)=\alpha_p\times \alpha_p$ and in the other case to
$\uGal(E/F)=\mu_p\times \mu_p$. Other iterable higher derivations
would also lead to different Galois group schemes. The truncated
Galois group scheme $G_t(E/F)$ only gives a bound on which Galois group
schemes $\uGal(E/F)$ may occur, because every one of them will be a
closed subgroup scheme of $G_t(E/F)$.


\vspace*{.5cm}

\parindent0cm


\begin{thebibliography}{MvdP03}

\bibitem[AM05]{amano_mas} Amano, K.. Masuoka, A.: {\it Picard-Vessiot
    extensions of artinian simple module algebras}. Journal of
  Algebra, 285:743-767 (2005)
\bibitem[Ama07]{amano} Amano, K.: {\it On a discrepancy among
    Picard-Vessiot theories in positive characteristics}. 
eprint arXiv:math/0612683v2 (2007) 
\bibitem[And01]{andre} Andr\'e, Y.: {\it Diff\'erentielles non
    commutatives et th\'eorie de Galois diff\'erentielle ou aux
    diff\'erences}.
Ann. Scient. \'Ec. Norm. Sup., (4), 34: 685-739 (2001)
\bibitem[BO78]{bert_ogus} Berthelot, P.. Ogus, A.: {\it Notes on
    crystalline cohomology}. Princeton University Press (1978)
\bibitem[Cha76]{chase} Chase, S.U.: {\it Infinitesimal Group Scheme
    Actions on Finite Field Extensions}. Am. J. Math., Vol. 98, No.2,
    pp. 441-480 (1976)
\bibitem[Del90]{deligne2} Deligne, P.: {\it Cat\'egories Tannakiennes},
in Grothendieck Festschrift, Vol II. 
Birkh\"auser Boston, pp. 111-195 (1990)
\bibitem[DG70]{demazure} Demazure, M.. Gabriel, P.: {\it Groupes
    Alg\'ebriques, tome 1}. North-Holland Pub. Comp., Amsterdam (1970)
\bibitem[DM89]{del_milne} Deligne, P.. Milne, J.: {\it Tannakian Categories},
in  Hodge Cycles, Motives and Shimura Varieties.
Springer Lecture Notes 900:101-228 (1989)
\bibitem[Dyc08]{dyckerhoff} Dyckerhoff, T.: {\it The Inverse Problem
    of Differential Galois Theory over the Field
    $\mathbb{R}(t)$}.  eprint arXiv:0802.2897v1 (2008)
\bibitem[Gro64]{ega} Grothendieck, A.: {\it \'El\'ements de
    g\'eom\'etrie alg\'ebrique IV}. Publ. Math. de l'I.~H.~\'E.~S. 20 (1964)
\bibitem[Eis95]{eisenbud} Eisenbud, D.: {\it Commutative Algebra with a View
  Toward Algebraic Geometry}.
Graduate Texts in Mathematics 150, Springer (1995)
\bibitem[Gie75]{gieseker} Gieseker, D.: {\it Flat vector bundles and
    the fundamental group in non-zero characteristics}. Ann. SNS Pisa,
  Classe di Scienze (4), 2, 1: 1-31 (1975)
\bibitem[Hard08]{hardouin} Hardouin, C.: {\it Iterative q-difference
    Galois Theory}. IWR-Preprint (2008)
(available at http://www.ub.uni-heidelberg.de/archiv/8278/ )
\bibitem[Hart77]{hartshorne} Hartshorne, R.: {\it Algebraic Geometry}.
Graduate Texts in Mathematics 52, Springer (1977)
\bibitem[HS37]{hasse_schmidt} Hasse, H.. Schmidt, F.K.: {\it Noch eine
Begr\"undung der Theorie des h\"o\-he\-ren Differentialquotienten in einem
algebraischen Funktionenk\"orper in einer Unbestimmten}.
J. Reine Angew. Math. 177:215-237 (1937)
\bibitem[Hei07]{heiderich} Heiderich, F.: {\it Picard-Vessiot-Theorie
f\"ur lineare partielle Differentialgleichungen}. Heidelberg
University Library, Diplom thesis (2007) 
\bibitem[Jac64]{jacobson} Jacobson, N.: {\it Lectures in Abstract
    Algebra, Vol III: Theory of Fields and Galois Theory}. Van
  Nostrand, Princeton, NJ (1964); Springer-Verlag reprint (1975)
\bibitem[Jan03]{jantzen} Jantzen, J.C.: {\it Representations of
    algebraic groups}. Am. Math. Soc. (2003)
\bibitem[Kat70]{katz2} Katz, N.: {\it Nilpotent connections and the
    monodromy theorem: applications of a result of
    Turrittin}. Publ. Math. de l'I.~H.~\'E.~S. 39:175-232 (1970)
\bibitem[Kat87]{katz} Katz, N.: {\it On the calculation of some differential
Galois groups}. Invent. Math. 87:13-61 (1987) 
\bibitem[Mat01]{mat_hart} Matzat, B.H.: {\it Differential
Galois Theory in Positive Characteristic}, notes written by
J. Hartmann. IWR-Preprint 2001-35 (2001)
(available at http://www.iwr.uni-heidelberg.de/$\sim$Heinrich.Matzat/publications.html)
\bibitem[Mats89]{matsumura} Matsumura, H.: {\it Commutative ring theory}.
Cambridge studies in adv. math. 8 (1989)
\bibitem[MvdP03]{mat_put} Matzat, B.H.. van der Put, M.: {\it Iterative
Differential Equations and the Abhyankar Conjecture}. J. Reine
Angew. Math. 557:1-52 (2003)
\bibitem[Pap08]{papanikolas} Papanikolas, M.: {\it Tannakian duality
    for Anderson-Drinfeld motives and algebraic independence of
    Carlitz logarithms}. Invent. Math. 171:123-174 (2008)
\bibitem[R\"os07]{roesch} R\"oscheisen, A.: {\it Iterative Connections
    and Abhyankar's Conjecture}. Heidelberg University Library, PhD
    thesis (2007) (available at http://www.ub.uni-heidelberg.de/archiv/7179/)
\bibitem[San07]{santos} Dos Santos, J.P.: {\it Fundamental group
    schemes for stratified sheaves}. Journal of Algebra 317:691-713 (2007)
\bibitem[Tak89]{takeuchi} Takeuchi, M.: {\it A Hopf Algebraic Approach
    to the Picard-Vessiot Theory}. Journal of Algebra 122:481-509
  (1989)
\bibitem[vdPS03]{put_singer} van der Put, M.. Singer, M.F.: {\it Galois
    Theory of Linear Differential Equations}. Grundlehren
  Math. Wiss. 328, Springer (2003)
\bibitem[Voj07]{vojta} Vojta, P.: {\it Jets via Hasse-Schmidt Derivations}.
In Diophantine Geometry, Proceedings, pp. 335-361. Edizioni della
Normale, Pisa (2007)
%
\end{thebibliography}
\end{document}